\documentclass[11pt]{amsart}
\usepackage{amssymb,mathrsfs,stmaryrd}	
\usepackage[all]{xy}
\usepackage{graphicx}
\usepackage[mathcal]{euscript}
\usepackage{verbatim}
\usepackage{tikz}
     \usetikzlibrary{positioning}
\usepackage{subcaption}

\usepackage[colorlinks,citecolor=blue]{hyperref}

\renewcommand{\and}{\qquad\text{and}\qquad}

\let\oldmarginpar\marginpar
\renewcommand\marginpar[1]{\-\oldmarginpar[\raggedleft\footnotesize #1]%
{\raggedright\footnotesize #1}}

\usepackage{mathtools}
\usepackage{manfnt}

\usepackage{color}

\definecolor{jade}{rgb}{0.10, 0.56, 0.42}
\definecolor{cerise}{rgb}{0.87, 0.19, 0.39}
\usetikzlibrary{arrows.meta}
\tikzset{>={latex[width=3mm,length=3mm]}}

\usetikzlibrary{arrows,backgrounds, fit, calc, positioning}
\usepackage{multirow}

\theoremstyle{definition}
\newtheorem*{thm*}{Theorem}
\newtheorem{thm}{Theorem}[section]
\newtheorem{definition}[thm]{Definition}
\newtheorem{conjecture}[thm]{Conjecture}
\newtheorem{lem}[thm]{Lemma}
\newtheorem{prop}[thm]{Proposition}
\newtheorem{corollary}[thm]{Corollary}
\newtheorem*{TDS}{Toral Deligne--Simpson Problem}
\newtheorem*{thm:mainthm}{Theorem \ref{thm:mainthm}}
\newtheorem*{thm:rigidthm}{Theorem \ref{thm:rigidthm}}
\theoremstyle{remark}
\newtheorem{rmk}[thm]{Remark}
\newtheorem{example}[thm]{Example}

\newcommand{\bl}[1]{\textcolor{blue}{#1}}

\newcommand{\N}{\Z_{>0}}
\newcommand{\Z}{\mathbb{Z}}

\newcommand{\C}{\mathbb{C}}

\newcommand{\pp}{\mathbb{P}^1}

\newcommand{\Ad}{\mathrm{Ad}}
\newcommand{\ad}{\mathrm{ad}}
\DeclareMathOperator{\ch}{char}
\newcommand{\sC}{\mathscr{C}}

\newcommand{\orb}{\mathscr{O}}
\newcommand{\rk}{\mathrm{rank}}

\newcommand{\cF}{\mathcal{F}}

\newcommand{\cB}{\mathcal{B}}
\newcommand{\cC}{\mathcal{C}}

\newcommand{\ftype}{\mathscr{A}}
\newcommand{\ftypes}{\mathcal{A}}
\newcommand{\DS}{\mathrm{DS}}

\newcommand{\bfA}{\mathbf{A}}
\newcommand{\bfO}{\mathbf{O}}
\DeclareMathOperator{\M}{\mathcal{M}}

\newcommand{\G}{{G}}
\DeclareMathOperator{\fg}{\mathfrak{g}}
\DeclareMathOperator{\fgl}{\mathfrak{gl}}
\DeclareMathOperator{\fsl}{\mathfrak{sl}}
\DeclareMathOperator{\GL}{GL}
\DeclareMathOperator{\SL}{SL}
\DeclareMathOperator{\Irr}{Irr}
\newcommand{\T}{T}
\newcommand{\B}{B}
\newcommand{\Stor}[1]{S^{{#1}}}
\newcommand{\symmetric}[1]{\mathfrak{S}_{{#1}}}

\newcommand{\Iwa}{I}
\newcommand{\Id}{\mathrm{id}}
\newcommand{\Galgp}{\mathcal{I}}

\newcommand{\Para}{P}

\newcommand{\ft}{\mathfrak{t}}

\newcommand{\fcox}{\mathfrak{c}}
\newcommand{\diag}{\mathrm{diag}}

\newcommand{\fs}{\mathfrak{s}}
\newcommand{\iwa}{\mathfrak{i}}
\newcommand{\fpara}{\mathfrak{p}}
\newcommand{\fp}{\mathfrak{p}}

\newcommand{\pow}{\mathfrak{o}}

\newcommand{\laur}{F}
\newcommand{\om}{\omega}

\newcommand{\Spec}{\mathrm{Spec}}
\newcommand{\nbr}[1]{[\fc]_{#1}}

\newcommand{\gc}{\nabla}
\newcommand{\fc}{\widehat{\gc}}

\newcommand{\prt}{\mathrm{Part}}
\newcommand{\prn}[2]{\lambda^{{#2},{#1}}}
\newcommand{\prp}[2]{\orb_{#2}^{#1}}
\newcommand{\orbp}[1]{\pi_{#1}}
\newcommand{\prinfilt}[1]{\langle{#1}\rangle}

\newcommand{\slope}{\mathrm{slope}}

\DeclareMathOperator{\Res}{\mathrm{Res}}
\DeclareMathOperator{\Tr}{Tr}
\DeclareMathOperator{\Lie}{\mathrm{Lie}}

\DeclareMathOperator{\im}{Image}
\newcommand{\Gm}{\mathbb{G}_m}

\newcommand{\la}[1]{\lambda_{#1}}
\newcommand{\m}[2]{\mu^{#1}_{#2}}
\newcommand{\flo}[1]{\lfloor{#1}\rfloor}

\newcommand{\dzz}{\frac{dz}{z}}
\newcommand{\tdzz}{\tfrac{dz}{z}}

\title{The Deligne--Simpson problem for connections on $\Gm$ with a maximally ramified singularity}

\author[M. Kulkarni]{Maitreyee C. Kulkarni}
\address{Mathematical Institute, University of Bonn, Bonn, Germany.}
\email{kulkarni@math.uni-bonn.de}
\author[N. Livesay]{Neal Livesay}
\address{Institute for Experiential AI, Northeastern University, Boston, MA.}
\email{n.livesay@northeastern.edu}
\author[J. Matherne]{Jacob P. Matherne}
\address{Mathematical Institute, University of Bonn, Bonn, Germany and Max Planck Institute for Mathematics, Bonn, Germany.}
\email{jacobm@math.uni-bonn.de}
\author[B. Nguyen]{Bach Nguyen}
\address{Department of Mathematics, Xavier University of Louisiana, New Orleans, LA.}
\email{bnguye22@xula.edu}
\author[D. S. Sage]{Daniel S. Sage}
\address{Department of Mathematics, Louisiana State University, Baton Rouge, LA.}
\email{sage@math.lsu.edu}


\begin{document}
  
\thanks{M.K. received support from the Charles Simonyi Endowment while
  at the Institute for Advanced Study, and from the Max Planck
  Institute for Mathematics in Bonn and the Hausdorff Research
  Institute for Mathematics in Bonn.  J.M. received support from NSF
  Grant DMS-1638352, the Association of Members of the Institute for
  Advanced Study, and the Hausdorff Research Institute for Mathematics
  in Bonn. B.N. received support from an AMS-Simons Travel Grant.
  D.S.S. received support from Simons Collaboration Grant 637367.}
 
\begin{abstract}
  The classical additive Deligne--Simpson problem is the existence
  problem for Fuchsian connections with residues at the singular
  points in specified adjoint orbits.  Crawley-Boevey found the
  solution in 2003 by reinterpreting the problem in terms of quiver
  varieties.  A more general version of this problem, solved by Hiroe,
  allows additional unramified irregular singularities.  We apply the
  theory of fundamental and regular strata due to Bremer and Sage to
  formulate a version of the Deligne--Simpson problem in which certain
  ramified singularities are allowed.  These allowed singular points
  are called toral singularities; they are singularities whose leading
  term with respect to a lattice chain filtration is regular
  semisimple.  We solve this problem in the special case of
  connections on $\Gm$ with a maximally ramified singularity at $0$
  and possibly an additional regular singular point at infinity.
  Examples of such connections arise from Airy, Bessel, and
  Kloosterman differential equations.  They play an important
  role in recent work in the geometric Langlands program.  We
  also give a complete characterization of all such connections which
  are rigid, under the additional hypothesis of unipotent monodromy at
  infinity.
 
\end{abstract}

\maketitle

\tableofcontents

\section{Introduction}
\subsection{The classical Deligne--Simpson problem}
A fundamental concern in the study of meromorphic connections is the
existence problem for connections with specified singularities.  More
precisely, this problem poses the question: given points $a_1,\dots,a_m$ in $\pp$ and formal
connections $\fc_1,\dots,\fc_m$, does there exist a meromorphic
connection $\nabla$ which is regular away from the $a_i$'s and
satisfies $\nabla_{a_i}\cong\fc_i$ for all $i$?  The classical
Deligne--Simpson problem is a variant of this problem for Fuchsian
connections.

From now on, we assume that the underlying vector bundles of all
connections on $\pp$ are trivializable.  Without loss of generality, we assume that the collection of singular
points does not include $\infty$. A Fuchsian connection with singular points
$a_1,\ldots,a_m$ is defined by\begin{equation*} d+\Big(\sum_{i=1}^m\frac{A_i}{z-a_i}\Big)dz,
\end{equation*}
where $A_i\in\fgl_n(\C)$ for all $i$.   Note that the adjoint orbit of $A_i$
determines the formal isomorphism class at $a_i$. Since $\infty$ is not a singularity, the
residue theorem forces $\sum A_i=0$.  We say that the collection of
matrices $A_1,\dots,A_m$ is irreducible if they have no common
invariant subspaces besides $\{0\}$ and $\C^n$. We can now state the (additive)
Deligne--Simpson problem:
\begin{quote}
\emph{Given adjoint orbits $\orb_1,\dots,\orb_m$,
  determine whether there exists an irreducible $m$-tuple
  $(A_1,\dots,A_m)$ with $A_i\in\orb_i$ satisfying $\sum
  A_i=0$~\cite{Kostov03}.}
  \end{quote}
    In other words, when is there an irreducible Fuchsian
connection with residues in the given orbits?  Note that the original problem
considered by Deligne and Simpson was the multiplicative version,
where one looks for Fuchsian connections with monodromies in specified
conjugacy classes in $\GL_n(\C)$~\cite{Simpson}.  The additive version stated above was
originally formulated by Kostov, who solved it in the ``generic''
case. Crawley-Boevey gave a complete solution by
reinterpreting the problem in terms of quiver varieties~\cite{CB03}. 
We remark that while there is an obvious analogue of this problem for
arbitrary reductive $\G$, little is known about the solution outside of type $A$.

\subsection{The unramified Deligne--Simpson problem}
In order to generalize the Deligne--Simpson problem to allow for
irregular singularities, one considers connections with higher
order principal parts at the singularities:
\begin{equation}\label{eq:pp} d+\Big(\sum_{i=1}^m \sum_{\nu=0}^{r_i}
  \frac{A^{(i)}_{\nu}}{(z-a_i)^{\nu}}\Big)\dzz.
\end{equation}
Again, we assume that $\infty$ is not a singular point, so
$\sum_{i=1}^m A^{(i)}_{0}=0$.  We now require that the singularity at
each $a_i$ has a certain specified form called a "formal type".

Most previous work on the irregular Deligne--Simpson problem has
restricted attention to the  "unramified
case"~\cite{Kostov10,Boa08,HirYam,Hiroe}.  This means that at each singularity, the slope decomposition of the corresponding formal
connection only involves integer slopes.  More
concretely, each such formal connection has Levelt--Turrittin (LT)
normal form \begin{equation}\label{LTunram} d+(D_rz^{-r}+\dots+ D_1z^{-1}+R)\dzz,
\end{equation}
where the $D_i$'s are diagonal, $D_r\ne 0$, and the residue term $R$
is upper triangular and commutes with each $D_i$.  We view the
$1$-form $(D_rz^{-r}+\dots+ D_1z^{-1}+R)\dzz$ as an unramified formal
type.  In the regular singular case, one can take $R$ to be in Jordan
canonical form and view the formal type as $R\dzz$.

For Fuchsian connections, the principal part at a singular point is
just the residue.  Hence, in the classical Deligne--Simpson problem,
one requires that the principal part agrees with the formal type after
conjugation by a constant matrix (i.e., an element of $\GL_n(\C)$).
In other words, the principal part lies in the adjoint orbit of the
formal type.  For unramified formal types of positive slope, one
instead requires the principal part to lie in the orbit of the formal
type under a certain action of the group $\GL_n(\C[\![z]\!])$.  Let
$\cB_r=\{(B_r z^{-r}+\dots + B_0)\dzz\mid B_i\in\fgl_n(\C)\}$ denote
the space of principal parts of order at most $r$.  The group
$\GL_n(\C[\![z]\!])$ acts on $\cB_r$ by conjugation followed by
truncation at the residue term.  Note that this action factors through
the finite-dimensional group $\GL_n(\C[\![z]\!]/z^r \C[\![z]\!])$.  If
$\ftype$ is an unramified formal type of slope $r$, we call the orbit
$\orb_{\ftype}$ under this action the \emph{truncated orbit} of
$\ftype$.  If $\ftype$ has slope $0$, $\orb_{\ftype}$ may be
identified with the usual adjoint orbit of $\ftype/\dzz$.

We can now state the unramified irregular Deligne--Simpson problem:
Given points $a_i$ and unramified formal types $\ftype_i$ of slope $r_i$,
determine when there exists an irreducible connection as in
\eqref{eq:pp} whose principal part at each $a_i$ lies in $\orb_{\ftype_i}$. 
This problem can also be restated in the language of moduli spaces.
Given the formal types $\ftype_i$, one can consider the moduli space of
``framable'' connections on a rank $n$ trivial bundle whose
singularities have the specified formal types~\cite{HirYam}.  The
construction generalizes that of Boalch~\cite{Boa}, who assumes that
the $\ftype_i$ are all nonresonant, i.e., that the leading term of each
$\ftype_i$ is regular semisimple.  This moduli space is not necessarily
well-behaved, but it is a complex manifold if one restricts to the
\emph{stable} moduli space, i.e., the open subset consisting of irreducible
connections.  The unramified Deligne--Simpson problem is simply the
question of when such a stable moduli space is nonempty.

This problem was solved in 2017 by Hiroe~\cite{Hiroe}, building on
earlier work of Boalch~\cite{Boa08} and Hiroe and
Yamakawa~\cite{HirYam}.  As in the Fuchsian case, the proof involves
quiver varieties.  Hiroe uses the collection of unramified formal
types to define a certain quiver variety and identifies the stable
moduli space with a certain open subspace of the quiver variety.  He
then finds necessary and sufficient conditions for this open subspace
to be nonempty.  As a corollary, Hiroe shows that the stable moduli
space is a connected manifold as long as it is nonempty.

\subsection{The ramified Deligne--Simpson problem for toral connections}\label{S:ramDS}
In this paper, we introduce the study of the \emph{ramified
  Deligne--Simpson problem}, where ramified singularities are allowed.
A singularity is called \emph{ramified} if the associated formal
connection can only be expressed in LT normal form after passing to a
ramified cover. The LT normal form is thus no longer a suitable notion
of formal type for ramified singularities.  It is possible to
formulate the ramified Deligne--Simpson problem by replacing the LT
normal form with a ``rational canonical form'' for connections.  Such a form may be obtained from Sabbah's refined Levelt--Turrittin decomposition~\cite{Sab08}; we
will discuss this in a future paper.

Here, we only sketch the setup of the ramified Deligne--Simpson problem for a special class of irregular
connections called \emph{toral connections}.  Roughly speaking, a
formal connection is called toral if its leading term with respect to
an appropriate filtration satisfies a graded version of regular
semisimplicity. (The precise definition involves the theory of
fundamental and regular strata for connections introduced by Bremer
and Sage~\cite{BrSa1,BrSa3,BrSa5}.)  The terminology reflects the fact
that toral connections can be ``diagonalized'' into a (not necessarily
split) Cartan subalgebra of the loop algebra.

First, we describe formal types for toral connections.  A rank $n$
toral connection has slope $r/b$, where $b$ is a divisor of $n$ or
$n-1$ and $\gcd(r,b)=1$.  If $b>1$, define
$\om_b\in\fgl_b(\C(\!(z)\!))$ by
$\om_b=\sum_{i=1}^{b-1}e_{i,i+1}+ze_{b,1}$; i.e., $\om_b$ is the
matrix with $1$'s in each entry of the superdiagonal, $z$ in the
lower-left entry, and $0$'s elsewhere. If $b=1$, set $\om_b=z$.  Note
that $\om_b^b=z\Id$.  Given such a $b$ with $b\ell=n$
(resp. $b\ell=n-1$), we define a block-diagonal Cartan subalgebra
$\fs^b=\C(\!(\om_b)\!)^\ell$ (resp.
$\fs^b=\C(\!(\om_b)\!)^\ell\oplus \C(\!(z)\!)$).  There is a natural
$\Z$-filtration $\fs^b=\bigcup_i(\fs^{b})^i$ induced by assigning
degree $i$ to $\om_b^i$.  Let $\Stor{b}$ denote the corresponding
maximal torus in the loop group.

An \emph{$\Stor{b}$-formal type} of slope $r/b$ (with $\gcd(r,b)=1$)
is a $1$-form $A\dzz$, where $A\in (\fs^{b})^{-r}$ has regular
semisimple term in degree $-r$ and no terms in positive degree.  It is
a fact that any toral connection of slope $r/b$ is formally isomorphic
to a connection $d+A\dzz$ with $A\dzz$ an $\Stor{b}$-formal type; the
formal type is unique up to an action of the relative affine Weyl
group of $\Stor{b}$~\cite{BrSa1,BrSa5}.

In the case of unramified toral connections, $\Stor{1}=T(\C(\!(z)\!))$
is the usual diagonal maximal torus, and the $\Stor{1}$-formal types
of slope $r$ are those connection matrices in LT normal form
\eqref{LTunram} with $D_r$ regular (so that $R$ is necessarily $0$).
At the opposite extreme, $\cC\coloneqq\Stor{n}=\C(\!(\om_n)\!)^*$ is a
"Coxeter maximal torus".\footnote{Under the bijection between classes
  of maximal tori in $\GL_n (\C(\!(z)\!))$ and conjugacy classes in
  the Weyl group $\symmetric{n}$~\cite{KaLu88}, $\cC$ corresponds to
  the Coxeter class consisting of $n$-cycles.}  The $\cC$-formal types
of slope $r/n$ are the $1$-forms $p(\om_n^{-1})\dzz$, where $p$ is a
polynomial of degree $r$.

The unramified Deligne--Simpson problem involves global connections
which satisfy a stronger condition than just having specified formal types
at the singularities.  One also needs the local isomorphisms transforming
the matrices of the formal connections into the given formal types to satisfy a global
compatibility condition called ``framability''.  We now explain how this
condition can be generalized to toral formal types.

Recall (see, e.g., \cite{Sa00,BrSa1}) that the \emph{parahoric
  subgroups} of $\GL_n(\C(\!(z)\!))$ are the local field analogues of
the parabolic subgroups of $\GL_n(\C)$.  A parabolic subgroup is the
stabilizer of a partial flag of subspaces in $\C^n$, and a parahoric
subgroup is the stabilizer of a "lattice chain" of
$\C[\![z]\!]$-lattices in $\C(\!(z)\!)^n$.  If $\Para$ is a parahoric
subgroup with associated lattice chain $\{L^j\}_j$, then there is an
associated "lattice chain filtration" $\{\fpara^{i}\}_{i\in\Z}$ on
$\fgl_n (\C(\!(z)\!))$ defined by
$\fpara^{i}=\{X\mid X(L^j)\subset L^{j+i} \text{ }\forall j\}$.

To each maximal torus $\Stor{b}$, there is a unique "standard
parahoric subgroup" $\Para^b\subset \GL_n(\C[\![z]\!])$ with the
property that the corresponding filtration $\{(\fpara^{b})^i\}_i$ is
compatible with the filtration on $\fs^b$, in the sense that
$(\fs^{b})^i=(\fpara^{b})^i\cap\fs^b$ for all $i$~\cite{BrSa1}.  In
the unramified case, we have $\Para^1=\GL_n(\C[\![z]\!])$.  For the
Coxeter maximal torus $\Stor{n}$, the corresponding parahoric subgroup
is the standard ``Iwahori subgroup'' $\Iwa\coloneqq\Para^n$; i.e.,
$\Iwa$ is the preimage of the upper-triangular Borel subgroup $\B$
(consisting of all upper-triangular matrices in $\GL_n(\C)$) via the
map $\GL_n(\C[\![z]\!])\to\GL_n(\C)$ induced by the ``evaluation at
zero'' map $z\mapsto 0$.

The most natural way of describing framability involves coadjoint
orbits.  One can view the principal part at $0$ of a connection as a
continuous functional on $\fgl_n(\C[\![z]\!])$ via
$Y\mapsto\Res(\Tr(YX\dzz))$.  Similarly, an $\Stor{b}$-formal type can be
viewed as a functional on $\fp^b$.  The global connection $\gc$ then
is \emph{framable at $0$} with respect to the $\Stor{b}$-formal type $A\dzz$ at
$0$ if for some global trivialization, the restriction to $\fpara^b$ of
the principal part at $0$ lies in the $\Para^b$-coadjoint orbit of
$A\dzz$.  (See Definition~\ref{D:formaltype}.)

However, one can also give a description more reminiscent of the
definition in the unramified case.

\begin{definition} \label{D:formaltypealt}
Let $\gc$ be a global connection on $\pp$ with a singular point at $0$,
and let $\ftype=A\dzz$ be a toral formal type of slope $r/b$. We say
that $\gc$ is \emph{framable at $0$ with respect to $\ftype$} if \begin{enumerate}\item  under some global trivialization $\phi$, the matrix form $\gc=d+[\gc]_\phi\dzz$ satisfies $[\gc]_\phi\in (\fpara^{b})^{-r}$, and $[\gc]_\phi-A\in
(\fpara^{b})^{1-r}$; and
\item there exists an element $p\in (\Para^{b})^1$ such
that the nonpositive
truncation of $\Ad(p)[\gc]_\phi$ equals $A$.
\end{enumerate}
\end{definition}

Recall that $\GL_n(\C)$ acts simply transitively on the space of global trivializations.  If one starts with a fixed trivialization $\phi^\prime$, then the choice of trivialization $\phi$ in the definition above corresponds to an element $g\in\GL_n(\C)$; i.e., there is a unique $g$ such that $\phi=g\cdot\phi^\prime$.  This matrix $g$ is called a \emph{compatible framing} (or
simply, a \emph{framing}) of $\gc$
at $0$.  Framability with respect to a formal type at an arbitrary point
$a\in\pp$ is defined
similarly, by simply replacing $z$ by $z-a$ if $a$ is finite, and by
$z^{-1}$ if $a=\infty$.

We can now state the Deligne--Simpson problem for connections whose
irregular singularities are all toral.  Note that the statement below can easily be extended to
allow for arbitrary unramified singular points.
\begin{TDS} Let $\bfA=(\ftype_1,\dots,\ftype_m)$ be a collection of toral formal
  types at the points $a_1,\dots,a_m\in\pp$, and let
  $\bfO=(\orb_1,\dots,\orb_\ell)$ be a collection of adjoint orbits at
 other points $b_1,\dots,b_\ell\in\pp$.  Does there exist an
  irreducible rank $n$ connection $\gc$ such that \begin{enumerate}\item $\gc$ is
    regular away from the $a_i$'s and $b_j$'s;
  \item $\gc$ is framable at $a_i$  with respect to the formal type $\ftype_i$; and
    \item $\gc$ is regular singular at $b_j$ with residue in $\orb_j$?
    \end{enumerate}
  \end{TDS} 
 \noindent  If such a connection exists, we call it a ``framable connection''
  with the given formal types.

  This problem can be restated in terms of moduli spaces of
  connections.  Suppose that each $\orb_j$ is nonresonant; i.e., suppose that no pair of the eigenvalues of the orbit differ by a nonzero integer.  Further assume
  that $m\ge 1$, so that there is at least one irregular singular
  point.  In \cite{BrSa1}, Bremer and Sage constructed the moduli
  space $\M(\bfA,\bfO)$ of connections satisfying all the above
  hypotheses except irreducibility.  Let
  $\M_{\mathrm{irr}}(\bfA,\bfO)$ be the subset of the moduli space
  consisting of irreducible connections.  In this language, the toral Deligne--Simpson
  problem poses the question of when $\M_{\mathrm{irr}}(\bfA,\bfO)$
  is nonempty.

  \subsection{Coxeter connections}

  We now restrict attention to a simple special case: connections with
  a \emph{maximally ramified} irregular singularity and (possibly) an
  additional regular singular point.  Without loss of generality, we
  will view such connections as connections on $\Gm$ with the
  irregular singularity at $0$.  Following \cite{KS2}, we refer to
  such connections as \emph{Coxeter connections}.  Well-known
  classical examples arise from the Airy differential equation and a
  modified version of the Bessel equation.\footnote{In these
    connections and others described below, the irregular singularity
    is at $\infty$.}  Another important class of examples consists of
  the generalized Kloosterman connections studied by
  Katz~\cite{KatzKloosterman,KatzExponential}. These hypergeometric connections are the
  geometric incarnations of certain exponential sums called
  Kloosterman sums, which are of great importance in number theory.

  Coxeter connections and their $\G$-connection
  analogues (for $\G$ a simple algebraic group) have played a
  significant role in recent work in the geometric Langlands program.
  For example, Frenkel and Gross~\cite{FrGr}, building on work of
  Deligne~\cite{D70} and Katz~\cite{KatzKloosterman,KatzRigid},
  constructed a rigid $\G$-connection of this type.  This connection,
  which may be viewed as a $\G$-version of a modified Bessel
  connection, was the first connection with irregular singularities
  for which the geometric Langlands correspondence was understood
  explicitly~\cite{HNY13,Zhu17}.  This connection also arises in Lam
  and Templier's proof of mirror symmetry for minuscule flag
  varieties~\cite{LamTemp17}. Other examples include the Airy
  $\G$-connection  and more general rigid ``Coxeter $\G$-connections''
  constructed in \cite{KS2}.  The Airy $\G$-connection and its $\ell$-adic analogue have also been studied in \cite{JKY21}.

Recall that if $\fc$ is a rank $n$ formal connection, then every slope
of $\fc$ has denominator (when the slope is expressed in lowest form)
between $1$ and $n$.  We say $\fc$ is maximally ramified if all such
denominators (or equivalently, at least one) is $n$.  In this case,
all the slopes are the same --- say $r/n$ with $\gcd(r,n)=1$ --- and
equal to the slope of the connection.  More concretely, the leading
term of the LT normal form is of the form $D_rz^{-r/n}$ with $D_r$ a
constant diagonal matrix, and is necessarily regular.

It is shown in \cite{KS3} that maximally ramified connections are
toral connections with respect to a Coxeter maximal torus. Thus any
maximally ramified connection of slope $r/n$ has a rational canonical
form $d+p(\om_n^{-1})\dzz$, where $p$ is a polynomial of degree $r$,
and the set of formal types is given by $\{p(\om_n^{-1})\dzz\mid p\in\C[x], \deg(p)=r\}$.
Moreover, any such connection is irreducible.

In this paper, we solve the ramified Deligne--Simpson problem for
Coxeter connections.  More precisely, let $\ftype$ be a maximally
ramified formal type, and let $\orb$ be an adjoint orbit (which we
will always assume to be nonresonant).  We determine necessary and sufficient conditions for the existence of a
meromorphic connection $\gc$ on $\pp$ which is framable at $0$ with
formal type $\ftype$, is regular singular with residue in $\orb$ at
$\infty$, and is otherwise nonsingular.  Note that any such connection
is automatically irreducible, since its formal connection at $0$ is
irreducible.  Thus, in the language of moduli spaces, we determine
when $\M(\ftype,\orb)$ is nonempty.

In order to state our result, we need some facts about adjoint orbits.   
Fix a monic polynomial
$q=\prod_{i=1}^s(x-a_i)^{m_i}$ of degree $n$ with the $a_i$'s distinct
complex numbers.  The set $\{X\in\fgl_n(\C)\mid \ch(X)=q\}$ of
matrices with characteristic polynomial $q$ is a
closed subset of $\fgl_n(\C)$ which is stable under conjugation.  We
denote the set of orbits with characteristic polynomial $q$ by $\orbp{q}$.  This set is partially ordered under the
usual Zariski closure ordering: $\orb\preceq\orb'$ if and only if
$\orb\subset\overline{\orb'}$.  The theory of the Jordan canonical
form makes it clear that $\orbp{q}$ can be identified with the
Cartesian product $\prod_{i=1}^s\prt(m_i)$, where $\prt(m_i)$ denotes the set of partitions of $m_i$.  Moreover, this identification defines a poset isomorphism between the closure ordering and the
direct product of the dominance orders.

Given positive integers $r$ and $m$, there exists a unique smallest
partition $\prn{r}{m}\in\prt(m)$ with at most $r$ parts.  Define
$\prp{r}{q}$ to be the orbit in $\orbp{q}$ corresponding to the
element
\[(\prn{r}{m_1}, \prn{r}{m_2}, \ldots, \prn{r}{m_s})\in\prod_{i=1}^s
  \prt(m_i).\] This tuple of partitions is the (unique) smallest
element of $\prod_{i=1}^s \prt(m_i)$ such that each component
partition has at most $r$ parts.  Note that $\prp{1}{q}$ is just the
regular orbit in $\orbp{q}$.  On the other extreme, if $r\ge m_i$ for
all $i$ (as is the case when $r\ge n$), then $\prp{r}{q}$ is the
semisimple orbit, the unique minimal orbit in $\orbp{q}$.  Let
$\prinfilt{\prp{r}{q}}$ denote the principal filter generated by
$\prp{r}{q}$ in $\orbp{q}$, i.e., $\orb\in\orbp{q}$ satisfies
$\orb\in\prinfilt{\prp{r}{q}}$ if and only if $\orb\succeq\prp{r}{q}$.
This filter is proper unless $r\ge m_i$ for all $i$.  As we will see
in Theorem~\ref{thm:r-regular}, the collection of orbits $\prp{r}{q}$
for each fixed $r$ satisfies a generalization of one characterization
of regular orbits.

We can now give the solution to the Deligne--Simpson problem for
Coxeter connections.  Given a rank $n$ maximally ramified formal type
$\ftype$ and a monic polynomial $q$ of degree $n$ that is nonresonant (i.e.,
no two roots differ by a nonzero integer), let
\[\DS(\ftype, q)=\{\orb\in\orbp{q}\mid \M(\ftype, \orb) \neq \varnothing
  \}.\]

\begin{thm:mainthm}\label{mainthm}
Let $r$ and $n$ be positive integers with $\gcd(r,n)=1$, let $\ftype$ be a maximally ramified
  formal type of slope $r/n$, and let $q=\prod_{i=1}^s(x-a_i)^{m_i}\in\C[x]$ with $a_1,\ldots,a_s\in\C$ distinct modulo $\Z$.  Then
  \[
    \DS(\ftype, q)=
    \begin{cases}
      \prinfilt{\prp{r}{q}} & \text{if }\Res(\Tr(\ftype))=-\sum_{i=1}^sm_ia_i, \\
      \varnothing & \text{else}.
    \end{cases}
  \]
\end{thm:mainthm}

In other words, given $\ftype=p(\om^{-1})\dzz$, then $\M(\ftype,\orb)$ is
nonempty if and only if $n p(0)=-\Tr(\orb)$ and $\orb\succeq
\prp{r}{\ch(\orb)}$.  Concretely, the condition $\orb\succeq
\prp{r}{\ch(\orb)}$ means that
$\orb$ has at most $r$ Jordan blocks for each eigenvalue.

Note that the solution depends only on the slope and the residue of
the formal type.

\begin{rmk} If $r\ge m_i$ for all $i$, then $\DS(\ftype, q)=\orbp{q}$
  as long as the trace condition is satisfied.  In particular, this is
  the case if $r>n$.
\end{rmk}

There is an obvious analogue of this problem for $\SL_n$-connections
(as opposed to $\GL_n$-connections).  Here, maximally ramified
formal types are of the form $p(\om^{-1})\dzz$ with $p(0)=0$ and
$\Tr(\orb)=0$.  Thus, the trace condition becomes vacuous, and the
Deligne--Simpson problem has a positive solution if and only if $\orb\succeq
\prp{r}{\ch(\orb)}$.

One can define Coxeter $\G$-connections for any simple group $\G$ (or
for any reductive group with connected Dynkin diagram)~\cite{KS2}.
For such a $\G$, Coxeter toral connections have slope $r/h$, where $h$
is the Coxeter number for $\G$ and $\gcd(r,h)=1$.  Moreover, there is
an analogue of the Deligne--Simpson problem in this more general
context.  We restrict to the case where the regular singularity at
$\infty$ has nilpotent residue (and thus has unipotent monodromy).

\begin{conjecture} Let $\G$ be a simple complex group with Lie algebra $\fg$.  Fix a Coxeter
  $\G$-formal type $\ftype$ of slope $r/h$ with $\gcd(r,h)=1$.  Then there
  exists a nilpotent orbit $\orb^r\subset\fg$ such that the
  Deligne--Simpson problem for Coxeter $\G$-connections with initial
  data $\ftype$ and the nilpotent orbit $\orb$ has a positive solution
  if and only if $\orb\succeq\orb^r$.  Moreover, if $r>h$, then
  $\orb^r=0$, so the Deligne--Simpson problem always has a positive
  solution.
\end{conjecture}

\subsection{Rigidity}
Our results have applications to the question of when Coxeter
connections are rigid.  Let $U\subset\pp$ be a nonempty open set, and
let $j:U\hookrightarrow\pp$ denote the inclusion.  A $\G$-connection
$\gc$ on $U$ is called \emph{physically rigid} if it is uniquely
determined by the formal isomorphism class at each point of
$\pp\setminus U$.  It is called \emph{cohomologically rigid} if
$H^1(\pp,j_{!*}\ad_{\gc})=0$.  For irreducible connections,
cohomological rigidity implies that $\gc$ has no infinitesimal
deformations.  For $\G=\GL_n(\C)$, it is a result of Bloch and Esnault
that cohomological and physical rigidity are the
same~\cite{BE04}.

In \cite{KS2}, Kamgarpour and Sage investigated the question of
rigidity for ``homogeneous'' Coxeter $\G$-connections with unipotent
monodromy.  A homogeneous Coxeter $\G$-formal type of slope $r/h$
(i.e., for $\GL_n$, a formal type of the form $a\om^{-r}_n\dzz$ with
$a\ne 0$) gives rise to a Coxeter $\G$-connection on $\Gm$ with
nilpotent residue at infinity.  They determined precisely when these
connections are (cohomologically) rigid, thus generalizing the work of Frenkel and
Gross~\cite{FrGr}. For $\GL_n$, it turns out that such connections are
rigid precisely when $r$ divides $n+1$ or $n-1$.

We can now generalize the results of \cite{KS2} to give a
classification of rigid Coxeter connections in type $A$.

\begin{thm:rigidthm}\label{rigidthm} Let $\ftype$ be a rank $n$ maximally ramified formal
  type of slope $r/n$, and let $\orb$ be any nilpotent orbit with
  $\orb\succeq \prp{r}{x^n}$.  Then there exists a rigid connection
  with the given formal type and and unipotent monodromy determined by
  $\orb$ if and only if $\orb=\prp{r}{x^n}$ and $r|(n\pm 1)$.
\end{thm:rigidthm}

We expect that the analogous statement is true for Coxeter $\G$-connections.

\subsection{Organization of the paper}

In \S\ref{prelim}, we discuss some facts about the poset of adjoint
orbits in $\fgl_n(\C)$ that will be needed in our applications. In
particular, we introduce and characterize a sequence of orbits which
generalize regular orbits.  In \S\ref{s:lattice}, we provide a brief
review of lattice chain filtrations.  In \S\ref{s:toral}, we describe
the role of these filtrations in studying formal connections,
following earlier work of Bremer and Sage~\cite{BrSa1, BrSa2, Sa17}.
In particular, we discuss toral connections and characterize maximally
ramified formal connections as Coxeter toral connections.  In
\S\ref{meromorphicconnections}, we describe moduli spaces of
connections with toral singularities and then state and prove our main
result on the Deligne--Simpson problem for Coxeter connections.  We
conclude the paper in \S\ref{s:rigidity} by characterizing rigid
Coxeter connections with unipotent monodromy at the regular singular
point.

\subsection*{Acknowledgements}The authors are deeply grateful to the American Institute of Mathematics (AIM) for hosting and generously supporting their research during SQuaRE meetings in 2020 and 2021.  Many of the key ideas in this paper were conceived during these meetings.  The authors would also like to thank an anonymous referee for helpful comments and suggestions that served to improve the manuscript.

\section{The poset of adjoint orbits}\label{prelim}

The solution to the Deligne--Simpson problem for Coxeter connections
involves certain distinguished orbits for the adjoint action of the
general linear group $\GL_n(\C)$ on $\fgl_n(\C)$ (i.e., similarity
classes of $n\times n$ complex matrices).  We will need some facts about
these adjoint orbits.

The set $\pi$ of adjoint orbits in $\fgl_n(\C)$ is partially ordered
via the closure order: $\orb\preceq\orb'$ if
$\orb\subseteq\overline{\orb'}$.  Let $\ch$ be the map sending a
matrix to its characteristic polynomial.  Given a monic degree $n$
polynomial $q$, $\ch^{-1}(q)$ is closed and $\GL_n(\C)$-stable.  If we
let $\orbp{q}$ be the set of adjoint orbits in $\ch^{-1}(q)$, it is
immediate that as a poset,
\begin{equation*} \pi=\bigsqcup_{\substack{\text{$q$ monic}\\
      \deg(q)=n}}\orbp{q}.
\end{equation*}

The theory of Jordan canonical forms allows us to identify the posets
$\orbp{q}$ with posets involving partitions.  If $n$ is a positive
integer, a \emph{partition of $n$} is a nonincreasing sequence
$\lambda=(\la{1},\la{2},\ldots,\la{m})$ of positive integers that sum
to $n$.  Each integer appearing in this sequence is called a part of
$\lambda$.  The total number of parts is denoted by $|\lambda|$.  It
will sometimes be convenient to use exponential notation for
partitions: if the $b_i$'s are the distinct parts, each appearing with
multiplicity $k_i$, we will also denote by $\lambda$ the multiset
$\{b_1^{k_1},\dots,b_s^{k_s}\}$.  Let $\prt(n)$ be the set of
partitions of $n$.  We view $\prt(n)$ as a poset via the dominance
order
  \begin{equation}\label{dom}
    \text{$\la{}\succeq\m{}{}\quad \iff\quad|\la{}|\le |\m{}{}|$
and $\sum_{i=1}^j\la{i}\ge\sum_{i=1}^j\m{}{i}$ for all
$j\in[1,|\la{}|]$}.
\end{equation}

Write $q=\prod_{i=1}^s(x-a_i)^{m_i}$ for distinct $a_1,\dots,a_s\in\C$
and $m_1,\ldots,m_s\in\N$.  The set $\orbp{q}$ can be identified with
$\prod_{i=1}^s\prt(m_i)$, where the partition of $m_i$ is given by the
sizes of the Jordan blocks with eigenvalue $a_i$.  It is well-known
that the closure order  corresponds to the product of the dominance
orders under this identification.

The unique maximal orbit in $\orbp{q}$ is the orbit with a single Jordan block for each eigenvalue.  This is
the regular orbit with characteristic polynomial $q$, i.e., the
unique orbit in $\orbp{q}$ of codimension $n$. 
We now define a sequence $\{\prp{r}{q}\}_r$
of orbits in $\orbp{q}$ which generalize the regular orbit.

Fix $r\in\N$, and consider the subset $\cF^r_q\subset\orbp{q}$
consisting of orbits with at most $r$ Jordan blocks for each
eigenvalue.  It is immediate from \eqref{dom} that $\cF^r_q$ is a filter
in the poset $\orbp{q}$.  This means that if $\orb\in\cF^r_q$ and $\orb\preceq\orb'$, then
$\orb'\in\cF^r_q$.  It turns out that $\cF^r_q$ is a principal filter;
i.e., there exists a (unique) element $\prp{r}{q}\in\orbp{q}$ such
that $\cF^r_q=\prinfilt{\prp{r}{q}}\coloneqq\{\orb\in\orbp{q}\mid \orb\succeq \prp{r}{q}\}$.  We will also set $\cF^r=\bigcup_q \cF^r_q$; it is a filter in $\pi$.  Note that $\cF^1$ is the set of all regular (or equivalently, maximal) adjoint orbits.

\begin{prop}  Given $n,r\in\N$, write $n=kr+r^{\prime}$ with
  $k,r^{\prime}\in\Z$ and $0\leq r^{\prime} < r$.  Then the partition
  \[\prn{r}{n}= \{(k+1)^{r^\prime}, k^{r-r^\prime}\}\] is the (unique)
  smallest partition of $n$ with at most $r$ parts. 
\end{prop}
\begin{proof}  Let $\lambda$ be any partition of $n$ with at most $r$
  parts, say with biggest part $u$ and smallest part $v$.  We will
  show that there exists a strictly smaller partition with at most $r$
  parts unless $\lambda=\prn{r}{n}$.

  If $r\ge n$, then $\prn{r}{n}=\{1^n\}$, the smallest element of
  $\prt(n)$, so the statement is trivial. We thus may assume that
  $r<n$.  First, suppose that $|\lambda|<r$.  It follows that $u\ge
  2$, so one obtains a strictly smaller partition with $|\lambda|+1\le
  r$ parts by replacing one $u$ with $u-1$ and adjoining a new part
  with value $1$.  We may thus assume without loss of generality that $|\lambda|=r$.

  Next, suppose that  $u-v\ge 2$.  Define a partition $\mu$ with the same
  parts as $\lambda$ except one $u$ is replaced by $u-1$ and one $v$
  is replaced by $v+1$.  It is obvious that $|\lambda|=|\mu|$ and that
  $\lambda$ is strictly bigger than $\mu$.

  It remains to consider the case $u-v\le 1$, so
  $\lambda=\{(v+1)^s,v^{r-s}\}$ for some $s$ with $0\le s <r$.  We
  then have $n=s(v+1)+(r-s)v=vr+s$, so $v=k$ and $s=r'$.  Thus,
  $\lambda=\prn{r}{n}$.
\end{proof}

Now, define $\prp{r}{q}$ to be the orbit in $\orbp{q}$ corresponding
to the element \[(\prn{r}{m_1}, \prn{r}{m_2}, \ldots,
  \prn{r}{m_s})\in\prod_{i=1}^s \prt(m_i).\]

\begin{corollary} The filter $\cF^r_q$ is principal with generator
  $\prp{r}{q}$, i.e., $\cF^r_q=\prinfilt{\prp{r}{q}}$.
\end{corollary}

\begin{proof}  It suffices to show that the corresponding filter in
  $\prod_{i=1}^s \prt(m_i)$ is principal with generator $(\prn{r}{m_1}, \prn{r}{m_2}, \ldots,
  \prn{r}{m_s})$.  This is immediate from the proposition.
\end{proof}

Recall that the semisimple orbit in $\orbp{q}$ is the unique minimal orbit.  If $r\ge m_i$ for all $i$ (as is the case when $r\ge n$),
then $\prp{r}{q}$ is the semisimple orbit so $\prinfilt{\prp{r}{q}}=\orbp{q}$.

We can now give a Lie-theoretic interpretation of $\cF^r$ which will be important in our applications.  Let $V^r$ be the set of matrices with nonzero entries on, and $0$'s below, the $r$th subdiagonal:
\begin{equation*} V^r=\{(x_{ij})\in\fgl_n(\C)\mid \text{$x_{ij}=0$ if $i-j>r$ and $x_{ij}\ne 0$ if $i-j=r$}\}.
\end{equation*}
Note that $V^r=\fgl_n(\C)$ if $r\ge n$.  It is well-known that every element of $V^1$ is regular and that every regular orbit has a representative in $V^1$.  (It is a famous result of Kostant that the analogous statement holds  for any complex simple group~\cite{Kos59}.)

We now prove a generalization of this result.
\begin{thm}\label{thm:r-regular} The adjoint orbits which intersect $V^r$ are precisely the orbits in $\cF^r$.  The minimal such orbits are the $\prp{r}{q}$'s.
\end{thm}

We begin by proving that $\prp{r}{q}$ intersects $V^r$.  Let $N_{r,n}$ be the $n\times n$ matrix with $1$'s
on the $r$th subdiagonal and $0$'s elsewhere.  We usually omit $n$
from the notation.

\begin{prop}\label{Jordanform}  Fix a positive integer $r<n$.  Let
$D_q=\diag(a_1,\dots,a_1,a_2,\dots,a_2,\dots,a_s,\dots,a_s)$, where
the eigenvalue $a_i$ appears with multiplicity $m_i$.  Let $U_r$ be
any matrix with all entries on the $r$th subdiagonal nonzero and all other
entries $0$.  Then $U_r+D_q\in \prp{r}{q}$.
\end{prop}
\begin{proof}  It is easy to see that there exists an invertible
  diagonal matrix $t$ such that $\Ad(t)(U_r+D_q)=N_r+D_q$, so we may
  assume without loss of generality that $U_r=N_r$.  We prove the
  proposition by induction on the number of distinct eigenvalues $s$
  (for arbitrary $n$).

  If $s=1$, the partition for the single
  eigenvalue is $\lambda^{n,r}$.  Indeed, the Jordan strings for
  $N_r+D_q-a_1\Id=N_r$ are given by $e_i\mapsto e_{i+r}\mapsto
  e_{i+2r}\mapsto\cdots$ for $1\le i\le r$, with the strings of length
  $k+1$ for $1\le i\le r'$ and $k$ otherwise.

  Now assume $s>1$.  Let $V$ be the span of $e_{m_1+1},\dots,e_n$.  Note
  that $N_r+D_q$ stabilizes $V$; in fact,
  $(N_r+D_q)|_V=N_{r,n-m_1}+D_{\hat{q}}$, where
  $\hat{q}=\prod_{i=2}^s(x-a_i)^{m_i}$.  It is clear that the Jordan
  strings for $(N_r+D_q)|_V$ are also Jordan strings for $N_r+D_q$ corresponding to
  the eigenvalues $a_2,\dots,a_s$. Also, $N_r+D_q$ induces the endomorphism
  $N_{r,m_1}+D_{(x-a_1)^{m_1}}$ on $\C^n/V$.  It remains to show that
      each Jordan string
      $\bar{e}_i\mapsto
      \bar{e}_{i+r}\mapsto\dots\mapsto\bar{e}_{i+(\ell-1) r}\mapsto 0$
      for $N_{r,m_1}$ on $\C^n/V$ lifts to a Jordan string for the
      zero eigenvalues of $N_r+D_q- a_1\Id$.  (Here, $\ell$ is the
      smallest integer such that $i+\ell r>m_1$.)  Let
      $f=(N_r+D_q- a_1\Id)|_V$.  This map is invertible, so
      $e_i-f^{-\ell}(e_{i+\ell r}) \mapsto\dots\mapsto e_{i+(\ell-1)
        r}-f^{-1}(e_{i+\ell r})\in\ker(N_r+D_q- a_1\Id)$ is the desired
      lift.
    \end{proof}
 \begin{proof}[Proof of Theorem~\ref{thm:r-regular}] 
   We first show that $V^r$ consists of elements whose Jordan forms
   have at most $r$ blocks for each eigenvalue.  Let $X\in V^r$, and
   let $a$ be any eigenvalue of $X$.  The bottom $n-r$ rows of the
   matrix $X-a\Id$ are linearly independent, so $\rk(X-a\Id)\ge n-r$.
   We conclude that $\dim\ker(X-a\Id)\le r$.  Since this dimension is
   the number of Jordan blocks for the eigenvalue $a$, the claim
   follows.
  
   We have shown that the set of adjoint orbits intersecting $V^r$ is
   contained in $\cF^r$.  Now take $\orb\in\cF^r$.  Let $q$ be its
   characteristic polynomial, so $\orb\succeq\prp{r}{q}$.  Take any
   $X\in\prp{r}{q}$ as in Proposition~\ref{Jordanform}.  Then
   $X\in V^r$.  By a theorem of Krupnik~\cite[Theorem 1]{krupnik97},
   there exists a strictly upper triangular matrix $Z\in\fgl_n(\C)$
   such that $X+Z\in\orb$. Since $X+Z\in V^r$, this proves the
   theorem.
\end{proof}

\begin{rmk} Krupnik's approach does not lead to explicit constructions of orbit representatives in $V^r$.  However, at least for nilpotent orbits, 
it is possible to find explicit, simple representatives by means of an algorithm described in~\cite{KLMNS2}.
\end{rmk}

\section{Lattice chain filtrations}\label{s:lattice}

Let $\laur=\C(\!(z)\!)$ denote the field of formal Laurent series, and
let $\pow=\C[\![z]\!]$ denote the ring of formal power series.  An
\emph{$\pow$-lattice} $L$ in $F^n$ is a finitely generated
$\pow$-module with the property that
$L\otimes_{\pow}\laur \cong \laur^n$.  A \emph{lattice chain} in
$\laur^n$ is a collection $\{L^i\}_{i\in\Z}$ of lattices satisfying
the following properties:
\begin{enumerate}
\item $L^i\supsetneq L^{i+1}$ for all $i$; and
\item there exists a positive integer $e$, called the \emph{period}, such that $L^{i+e} = zL^i$ for all $i$.
\end{enumerate}
A \emph{parahoric subgroup} $\Para\subset\GL_n(\laur)$ is the
stabilizer of a lattice chain; i.e.,
$\Para=\{g\in\GL_n(\laur) \mid gL^i = L^i\text{ for all }i\}$.  We
write $e_{\Para}$ to denote the period of the lattice chain stabilized
by $\Para$.  A lattice chain
in $F^n$ is called \emph{complete} if its period is as large as
possible, i.e., if its period equals $n$.  The parahoric subgroups
associated to complete lattice chains are called \emph{Iwahori
  subgroups}.

Each lattice chain $\{L^i\}_{i\in\Z}$ --- say with corresponding
parahoric $\Para$ --- determines a filtration $\{\fpara^i\}_{i\in\Z}$
of $\fgl_n(\laur)$ defined by
$\fpara^i = \{X\in\fgl_n(\laur) \mid XL^j \subset L^{j+i}\text{ for
  all }j\}$; in particular, $\fpara^0=\fpara\coloneqq\Lie(\Para)$.
One also gets a filtration $\{\Para^i\}_{i\in\Z_{\ge 0}}$ of $\Para$
defined by $\Para^0=\Para$ and $\Para^i=1+\fpara^i$ for $i>0$.  In the
special case that each lattice in the lattice chain stabilized by
$\Para$ admits an $\pow$-basis of the form $\{z^{k_j}e_j\}_{j=1}^n$,
the corresponding filtration $\{\fpara^i\}_i$ is induced by a grading
$\{\fpara(i)\}_i$ on $\fgl_n(\C[z,z^{-1}])$.

\begin{rmk}\label{building} The ``lattice chain filtrations'' described above may also be obtained through a more general construction. Let $\cB$ be the (reduced) Bruhat--Tits building associated to the loop group $\GL_n(\laur)$.  This is a simplicial complex whose simplices are in bijective correspondence with the parahoric subgroups of the loop group.  
For each point in the building, there is an associated
``Moy--Prasad filtration'' on the loop algebra~\cite{MP}.  Up
to rescaling,
the lattice chain filtration determined by a parahoric subgroup $\Para$ is the Moy--Prasad
filtration associated to the barycenter of the simplex in the
Bruhat--Tits building corresponding to $\Para$~\cite{BrSa3}.

The building $\cB$ is the union of $(n-1)$-dimensional real affine spaces called apartments, which are in one-to-one correspondence with split maximal tori in $\GL_n(\laur)$. The parahoric subgroups described in the "special case" above --- i.e., the parahoric subgroups $\Para$ where each lattice in the lattice chain stabilized by $\Para$ admits an $\pow$-basis of the form
$\{z^{k_j}e_j\}_{j=1}^n$ --- are precisely the parahoric subgroups corresponding to the simplices in the "standard apartment", i.e., the apartment corresponding to the diagonal torus.
\end{rmk}

We will focus our attention on two particular parahoric subgroups of $\GL_n(\laur)$: $\GL_n(\pow)$ and the "standard Iwahori subgroup" $\Iwa$ (defined below).  The parahoric subgroup $\GL_n(\pow)$ is the stabilizer of the $1$-periodic lattice chain $\{z^i\pow^n\}_{i}$.
The associated filtration on $\fgl_n(\laur)$ is the \emph{degree
  filtration} $\{z^i\fgl_n(\pow)\}_{i\in\Z}$.
  
  The \emph{standard Iwahori subgroup} $\Iwa$ is the stabilizer of the \emph{standard
  complete lattice chain} in $F^n$, i.e., the lattice chain
\[\{\pow\text{-span}\{z^{\flo{\frac{i}{n}}}e_1,
z^{\flo{\frac{i+1}{n}}}e_2, \ldots,
z^{\flo{\frac{i+(n-1)}{n}}}e_n\}\}_{i\in\Z},\] where $\{e_j\}_{j=1}^n$
is the standard basis for $F^n$.  The associated filtration is the
\emph{standard Iwahori filtration} $\{\iwa^i\}_{i\in\Z}$.  Note that if
$\B\subset\GL_n(\C)$ is the Borel subgroup of invertible upper
triangular matrices, then $\Iwa$ is the preimage of $\B$ under the homomorphism
$\GL_n(\pow)\to\GL_n(\C)$ induced by the "evaluation at zero" map $z\mapsto 0$.

Define $\om\in\fgl_n(\laur)$ by
$\om=\sum_{i=1}^{n-1}e_{i,i+1}+ze_{n,1}$; i.e., $\om$ is the
matrix with $1$'s in each entry of the superdiagonal, $z$ in the
lower-left entry, and $0$'s elsewhere.  Then
$\om^i\iwa^j=\iwa^j\om^i=\iwa^{i+j}$ \cite[Proposition 1.18]{Bu}
for all $i,j\in\Z$.  Note that $\iwa^{i+n}=z\iwa^i$ for all $i$. Let
$\T$ denote the standard diagonal maximal torus in $\GL_n(\C)$.  Then $\ft=\Lie(\T)$ consists of all diagonal matrices in
$\fgl_n(\C)$. The standard Iwahori grading $\{\iwa(i)\}_{i\in\Z}$ is
then given by \begin{equation}\label{iwahorigrading}\iwa(i)=\om^i\ft=\ft\om^i.\end{equation}

\begin{example}
  Let $n=3$.  Then
  \[\om = \begin{bmatrix} 0 & 1 & 0 \\ 0 & 0 & 1 \\ z & 0 & 0 \end{bmatrix}\and\Iwa=\begin{bmatrix} \pow^* & \pow & \pow \\ z\pow & \pow^* & \pow \\ z\pow & z\pow & \pow^* \end{bmatrix}.\]
  Some steps in the standard Iwahori filtration on $\fgl_n(\laur)$ are shown below:
  \[
    \iwa^{-2}=\begin{bmatrix} \pow & z^{-1}\pow & z^{-1}\pow \\ \pow &
      \pow & z^{-1}\pow \\ \pow & \pow &
      \pow \end{bmatrix}\quad\subsetneq\quad \iwa^{-1}=\begin{bmatrix}
      \pow & \pow & z^{-1}\pow \\ \pow & \pow & \pow \\ z\pow & \pow &
      \pow \end{bmatrix}\quad\subsetneq\quad \iwa^{0}=\begin{bmatrix}
      \pow & \pow & \pow \\ z\pow & \pow & \pow \\ z\pow & z\pow &
      \pow \end{bmatrix}.\]
\end{example}

If $\mathfrak{h}\subset\fgl_n(\laur)$ is any $\C$-subspace containing
some $\fp^i$, we denote the space of continuous linear functionals on
$\mathfrak{h}$ by $\mathfrak{h}^\vee$. Every continuous functional on  $\mathfrak{h}$ extends to a continuous functional on $\fgl_n(\laur)$, so $\mathfrak{h}^\vee\cong\fgl_n(\laur)^\vee/\mathfrak{h}^\perp$, where $\mathfrak{h}^\perp$ denotes the annihilator of $\mathfrak{h}$. The space of $1$-forms $\Omega^1(\fgl_n(\laur))$ can be identified with the space of functionals $\fgl_n(\laur)^\vee$ by associating a $1$-form $\nu$ with the functional $Y\mapsto\Res(\Tr(Y\nu))$. This identification is well-behaved with
respect to lattice chain filtrations~\cite[Proposition 3.6]{BrSa3}:
\begin{equation*} (\fpara^i)^\vee\cong
  \fgl_n(\laur)\tfrac{dz}{z}/\fpara^{-i+1}\tfrac{dz}{z}.
\end{equation*}
When the filtration comes from a grading, one can be even more
explicit.  In particular, we
get  \begin{equation}\label{duals}\fgl_n(\pow)^\vee\cong\fgl_n(\C[z^{-1}]) \tfrac{dz}{z}\qquad
  \text{and} \qquad \iwa^\vee\cong\ft[\om^{-1}]\tfrac{dz}{z}.
\end{equation}

For applications to connections, it is important to consider the
relationship between filtrations on Cartan subalgebras of
$\fgl_n(\laur)$ and filtrations on parahoric subalgebras
~\cite{BrSa1,BrSa5}.

Note that, since $\laur$ is
not algebraically closed, it is not true that all maximal tori (or
equivalently, all Cartan subalgebras) are conjugate.  In fact, there
is a bijection between the set of conjugacy classes of maximal tori in $\GL_n(\laur)$ and the set of
conjugacy classes in the Weyl group for $\GL_n(\C)$ (i.e., the symmetric group $\symmetric{n}$)~\cite[Lemma 2]{KaLu88}.  We also remark that each Cartan subalgebra $\fs$ comes equipped with a natural filtration $\{\fs^i\}_{i\in\Z}$ (see, e.g., \cite[Section 3]{BrSa5}).  In the cases of interest to us in this paper, the filtration will be induced by a grading $\{\fs(i)\}_{i\in\Z}$.

Let $\Stor{}$ be a maximal torus and let $\Para$ be a parahoric subgroup, both in $\GL_n(\laur)$.  Let $\fs=\Lie(\Stor{})$ be the Cartan subalgebra associated to $\Stor{}$.  We say that $\Stor{}$ and $\Para$ (or $\fs$ and $\fpara$) are
\emph{compatible} (resp. \emph{graded compatible}) if
$\fs^i=\fpara^i\cap\fs$ (resp. $\fs(i)=\fpara(i)\cap\fs$) for all
$i$.  For present purposes, it suffices to consider two examples: the diagonal subalgebra $\ft(\laur)$ and the
``standard Coxeter Cartan subalgebra'' $\fcox=\C(\!(\om)\!)$.  The
diagonal Cartan subalgebra $\ft(\laur)$ (corresponding to the trivial class in
the Weyl group $\symmetric{n}$) is endowed with a filtration which comes
from the obvious grading $\ft(\C[z,z^{-1}])=\bigoplus_i z^i\ft$.  It is
immediate that $\T(\laur)$ is graded compatible with $\GL_n(\pow)$.

At the opposite extreme, there is a unique class of maximal tori in
$\GL_n(\laur)$ that are anisotropic modulo the center, meaning that
they have no non-central rational cocharacters.  Concretely, such tori
are as far from being split as possible.  This class corresponds to
the Coxeter class in $\symmetric{n}$, i.e., the class of $n$-cycles.  A specific
representative of this class is the standard Coxeter torus
$\cC=\C(\!(\om)\!)^*$ with Lie algebra $\fcox$.  Note that $\om$
is regular semisimple --- its eigenvalues are the $n$ distinct $n$th
roots of $z$ --- so its centralizer $\C(\!(\om)\!)$ is indeed
a Cartan subalgebra.  The natural grading by powers of $\om$ on
$\C[\om,\om^{-1}]$ induces a filtration on $\fcox$, and it is clear
that $\cC$ is graded compatible with $\Iwa$.

\section{Toral and maximally ramified connections}\label{s:toral}
\subsection{Formal connections}
A \emph{formal connection of rank $n$} is a connection $\fc$ on an
$\laur$-vector bundle $V$ of rank $n$ over the formal punctured disk
$\Spec(\laur)$.  Given a trivialization $\phi$ for $V$ (which is always trivializable), the connection
can be written in matrix form as $\fc=d+\nbr{\phi}$, where
$\nbr{\phi}\in\Omega^1_{\laur}(\fgl_n(\laur))$.  The loop group
$\GL_n(\laur)$ acts simply transitively on the set of trivializations
via left multiplication.  The corresponding action of $\GL_n(\laur)$
on the connection matrix is given by the \emph{gauge action}: if
$g\in\GL_n(\laur)$, then $g\cdot\nbr{\phi} = \nbr{g\cdot\phi} =
\Ad(g)(\nbr{\phi})-(dg)g^{-1}$.  Hence, the set of isomorphism classes
of rank $n$ formal connections is isomorphic to the orbit space
$\fgl_n(\laur)\dzz/\GL_n(\laur)$ for the gauge action.

A formal connection $\fc$  is called \emph{regular singular} if the
connection matrix with respect to some trivialization has
a simple pole.  If the matrix has a higher order pole for every
trivialization, $\fc$ is said to be
\emph{irregular singular}.  Katz defined an invariant of formal
connections called the \emph{slope} which gives one measure of the
degree of irregularity of a formal connection~\cite{D70}.  The slope is a
nonnegative rational number whose denominator in lowest form is at
most $n$.  The slope is positive if and only if $\fc$ is irregular.

\subsection{Fundamental strata}
The classical approach to the study of formal connections involves an analysis of the ``leading term'' of the connection matrix with respect to the degree filtration on $\fgl_n(\laur)$ \cite{Was}.  To review, suppose that the matrix for $\fc$ with respect to $\phi$ is expanded with respect to the degree filtration on $\fgl_n(\laur)$; i.e., suppose
\begin{equation}\label{degreeexpansion}[\fc]_\phi=(M_{-r}z^{-r}+M_{-r+1}z^{-r+1}+\cdots)\tdzz,\end{equation}
where $r\ge 0$ and $M_i\in\fgl_n(\C)$ for all $i$.  When the
\emph{leading term} $M_{-r}$ is well-behaved, it gives useful
information about the connection.  For example, if $M_{-r}$ is
non-nilpotent, then $\slope(\fc)=r$.  Moreover, if $r>0$ and $M_{-r}$
is diagonalizable with distinct eigenvalues, then $\fc$ can be
diagonalized into a ``$\T(\laur)$-formal type of depth $r$''.  This
means that there exists $g\in\GL(\laur)$ such that $[\fc]_{g\phi}$ is
an element of
\[\ftypes (\T(\laur),r)\coloneqq\{(D_{-r}z^{-r}+\cdots+D_1z+D_0)\tdzz
  \mid \forall i,D_i\in\ft \text{ and }D_{-r}\text{ has distinct
    eigenvalues}\}\] \cite{Was}.  Note that many interesting
connections have nilpotent leading terms.  For example, the leading term of
the formal Frenkel--Gross connection
$\fc_{\text{FG}}=d+\om^{-1}\dzz$ \cite{FrGr} is strictly upper
triangular (and thus nilpotent).  In fact, the leading term is
nilpotent no matter what trivialization one chooses for $V$.

More recently, Bremer and Sage --- borrowing well-known tools from
representation theory developed by Bushnell \cite{Bu}, Moy--Prasad
\cite{MP}, and others --- have introduced a more general approach to
the study of formal connections, where leading terms are replaced by
``strata'' \cite{BrSa1,BrSa5,BrSa3}.  A \emph{$\GL_n$-stratum} is a
triple $(\Para, r, \beta)$ with $\Para\subset \GL_n(\laur)$ a
parahoric subgroup, $r$ a nonnegative integer, and $\beta$ a
functional on $\fpara^r/\fpara^{r+1}$.  Consider the special case
where $\Para$ corresponds to a simplex in the standard apartment (see
Remark~\ref{building}).  Here, a functional
$\beta\in(\fpara^r/\fpara^{r+1})^\vee$ can be written uniquely as
$\beta^\flat\dzz$ for $\beta^\flat$ homogeneous (i.e., for
$\beta^\flat\in\fpara(-r)$). The stratum is called \emph{fundamental}
if $\beta^\flat$ is non-nilpotent.  A formal connection $\fc$
\emph{contains the stratum $(\Para, r, \beta)$} (with respect to a
fixed trivialization) if $\fc = d + X\dzz$ with $X\in\fpara^{-r}$ and
$\beta$ induced by $X\dzz$.  More general definitions of fundamental
strata and stratum containment are given in \cite{BrSa1, BrSa3}.

Fundamental strata can be viewed as a generalization of the notion of
a non-nilpotent leading term.  In particular, fundamental strata can
be used to compute the slope of any connection, not merely those with
integer slopes.  Recall that if $\Para\subset\GL_n(\laur)$ is a
parahoric subgroup, then $e_{\Para}$ denotes the period of the lattice
chain stabilized by $\Para$.

\begin{thm} [{\cite[Theorem 4.10]{BrSa1}, \cite[Theorem
    1]{Sa17}}] \label{T:slope} Any formal connection $\fc$ contains a
  fundamental stratum.  If $\fc$ contains the fundamental stratum
  $(\Para, r, \beta)$, then $\slope(\fc)=r/e_{\Para}$.\end{thm}

We now investigate some examples.  The connection in
(\ref{degreeexpansion}) (with $M_i\in\fgl_n(\C)$ for all $i$) contains
the stratum $(\GL_n(\pow), r, M_{-r}z^{-r}\dzz)$, which is fundamental
if and only if $M_{-r}$ is non-nilpotent.  The formal Frenkel--Gross
connection $\fc_{\text{FG}}$ contains the fundamental stratum
$(\Iwa, 1, \om^{-1}\dzz)$.  Moreover, any rank $n$ formal
connection of the form $\fc=d+(a\om^{-r}+X)\dzz$, with
$\gcd(r,n)=1$, $a\in\C^*$, and $X\in\iwa^{-r+1}$, contains the
fundamental stratum $(\Iwa,r,a\om^{-r}\dzz)$, and thus has slope
$r/n$.

\subsection{Toral connections}
The notion of a diagonalizable leading term with distinct eigenvalues
is generalized by the notion of a ``regular stratum''.  For
simplicity, we only consider the case where $\Para$ comes from the
standard apartment (see Remark~\ref{building}).  General definitions can be found in
\cite{BrSa1} and \cite{BrSa5}.  For such a $\Para$, consider the stratum $(\Para,r,
\beta^\flat\dzz)$.  If $\beta^\flat$ is regular semisimple, then its
centralizer $C(\beta^\flat)$ is regular semisimple, and we say that  $(\Para,r,
\beta^\flat\dzz)$ is a $C(\beta^\flat)$-regular stratum.  A connection
that contains an $\Stor{}$-regular stratum is called an \emph{$\Stor{}$-toral connection}.

It turns out that toral connections do not exist for every maximal
torus $\Stor{}$.  In fact, an $\Stor{}$-toral connection of slope $s$
exists if and only if $\Stor{}$ corresponds to a regular conjugacy
class in $\symmetric{n}$ (in the sense of Springer~\cite{Spr74}) and
if $e^{2\pi i s}$ is a regular eigenvalue of this conjugacy
class~\cite{BrSa5}.  The regular classes are parametrized by the
partitions $\{b^{n/b}\}$ (for positive divisors $b$ of $n$) and
$\{b^{(n-1)/b},1\}$ (for positive divisors $b$ of $n-1$).
Representatives for each of the corresponding conjugacy classes of
maximal tori are given by the $\Stor{b}$'s defined in \S\ref{S:ramDS}.
An $\Stor{b}$-toral connection has slope $r/b$ for some $r>0$ with
$\gcd(r,b)=1$.  Note that $\Stor{1}$ is the diagonal torus $\T(\laur)$
while $\Stor{n}$ is the standard Coxeter torus $\cC$.

Just as for connections whose naive leading term is regular
semisimple, there exist ``rational canonical forms'' for toral
connections involving the notion of a formal type.  Fix a divisor $b$
of either $n$ or $n-1$.  We define the set of \emph{$\Stor{b}$-formal
  types of depth $r$} (with $\gcd(r,b)=1$) by
\begin{equation*} \ftypes(\Stor{b},r)=\{A\tdzz\mid A=\sum_{i=0}^r A_{-i}\in
  \bigoplus_{i=0}^r \fs^b(-i)\text{ with $A_{-r}$ regular semisimple}\}.
\end{equation*}
Every toral connection $\fc$ of slope $r/b$ is formally isomorphic to a
connection of the form $d+A\tdzz$ with $A\tdzz\in \ftypes(\Stor{b},r)$; we view
this as a rational canonical form for $\fc$.

We will need a more precise variation of this
statement.  As mentioned in \S\ref{S:ramDS}, for each $b$, there is a
standard parahoric subgroup $\Para^b$ which is compatible (in fact, graded
compatible) with $\Stor{b}$.
\begin{thm}[{\cite[Theorem 4.13]{BrSa1}, \cite[Theorem 5.1]{BrSa5}}]\label{T:diag}
  Suppose $\fc$ is a formal connection containing an $\Stor{b}$-regular
  stratum $(\Para^b, r,\beta^\flat\dzz)$ with
  $\beta^\flat\in\fs^b(-r)$.  Then there exists $p\in\Para^{b,1}$ such
  that $p\cdot\nbr{}$ is a formal type in $\ftypes(\Stor{b},r)$ whose
  component in $\fs^b(-r)\dzz$ is $\beta^\flat\dzz$.
\end{thm}

\subsection{Maximally ramified connections}

\begin{definition}  A formal connection $\fc$ of rank $n$ is called
  \emph{maximally ramified} if it has slope $r/n$ with $\gcd(r,n)=1$.
\end{definition}

Thus, a formal connection is maximally ramified if the denominator of
the slope (in lowest terms) is as big as possible.  Another
interpretation involves the slope decomposition of $\fc$. It is a well-known result of Turrittin~\cite{Tur55} and Levelt~\cite{Lev75} that after extending scalars to $\C(\!(z^{1/b})\!)$ for some $b\in\N$, there exists a trivialization in which the matrix of $\fc$ is block-diagonal: 
\begin{equation*}
\fc=d+\diag(p_1(z^{-1/b})\Id_{m_1}+R_1,\dots,p_k(z^{-1/b})\Id_{m_k}+R_k)\tdzz;
\end{equation*}
here, the $p_i$'s are polynomials and the $R_i$'s are nilpotent matrices.  This is the Levelt--Turrittin normal form of $\fc$.   The \emph{slopes} of $\fc$ are the $n$ rational numbers $\deg(p_i)/b$, each appearing with multiplicity $m_i$.  This collection of invariants gives more detailed information about how irregular $\fc$ is than the single invariant $\slope(\fc)$.  Indeed, one can define $\slope(\fc)$ to be the maximum of the
slopes of $\fc$.

One can show that the slopes are  nonnegative rational numbers with
denominators at most $n$.  Moreover, if at least one slope is $r/n$ with
$\gcd(r,n)=1$, then all slopes are $r/n$.  Thus, $\fc$ is
maximally ramified if at least one slope has denominator $n$.

It turns out that maximally ramified connections are the same thing as
Coxeter toral connections.  If we specialize our results on formal
types to the standard Coxeter torus $\cC$, we see that the
$\cC$-formal types of depth $r$ are given by
\begin{equation*}
\ftypes (\cC,r)=\{p(\om^{-1})\tdzz \mid
p\in\C[x], \deg(p)=r\}.
\end{equation*}

We thus obtain the following result on rational canonical forms for maximally ramified connections:

\begin{thm}[{\cite{KS3}}]  Let $\fc$ be a maximally ramified connection
    of slope $r/n$ with $\gcd(r,n)=1$.  Then $\fc$ is formally gauge
    equivalent to a connection of the form $d+p(\om^{-1})\dzz$ with
    $p$ a polynomial of degree $r$.
  \end{thm}
  
  This theorem may be obtained as a direct corollary of Sabbah's refined Levelt--Turrittin decomposition~\cite[Corollary 3.3]{Sab08}.  Indeed, since $\fc$ has slope $r/n$, Sabbah's theorem shows that it is formally isomorphic to a connection of the form $[n]_*(d+d\phi+\lambda)$, where $[n]:\Spec(\laur) \to \Spec(\laur)$ is the $n$-fold covering induced by $z\mapsto z^n$, $\phi\in z^{-1}\C[z^{-1}]$ has degree $r$, and $\lambda\in\C$.  It is now easy to conclude that $\fc$ has the desired rational canonical form, with the coefficients of $p$ determined by $\lambda$ and the coefficients of $\phi$.
  
  However, the theory of toral connections allows one to prove a generalized version of this theorem for $\G$-connections, where  $\G$ is a reductive group with connected
  Dynkin diagram~\cite{KS3}.  Here, $n$ is replaced by the Coxeter number $h$,
  $\cC$ is an appropriate fixed Coxeter torus in $\G(\laur)$,
  and $\fc$ is formally isomorphic to $d+\ftype$, where $\ftype$ is a
  $\cC$-formal type of slope $r/h$.  Below, we provide a concise
  stratum-theoretic proof for the specific case $\G=\GL_n$, which is simpler than the general proof.

  \begin{proof} By Theorem~\ref{T:slope}, $\fc$ contains a fundamental
    stratum $(\Para,r,\beta')$ with respect to some trivialization.  Since
    $\slope(\fc)=r/e_{\Para}$, it follows that $e_{\Para}=n$ and $\Para$ is an Iwahori
    subgroup.  By equivariance of stratum containment and the fact
    that Iwahori subgroups are all conjugate, we may modify the
    trivialization so that $\fc$ contains the fundamental stratum
    $(\Iwa,r,\beta)$.  The functional $\beta$ is represented by
    $\beta^\flat\dzz$, where $\beta^\flat\in\iwa(-r)$ is
    non-nilpotent.

    By \eqref{iwahorigrading}, $\beta^\flat=\diag(a_1,\dots,a_n)\om^{-r}$ for
    some constants $a_i$.  Let $a=\prod a_i$.  Since
    $\ch(\beta^\flat)=x^n- az^{-r}$, it follows that $a\ne 0$.  The
    polynomial thus has distinct roots, and $\beta^\flat$ is regular
    semisimple.  Let $a^{1/n}$ be a fixed $n$th root of $a$.  It is
    easy to see that there exists $t\in \T$ such that
    $\Ad(t)(\beta^\flat)=a^{1/n}\om^{-r}$, and since $t$ normalizes $\Iwa$, $\fc$ contains the stratum  $(\Iwa,r, a^{1/n}\om^{-r}\dzz)$. By
    Theorem~\ref{T:diag}, $\fc$ is formally isomorphic to
    $d+(b_r\om^{-r} + \dots + b_0)\dzz$, where $b_r=a^{1/n}$.
  \end{proof}

  \section{The Deligne--Simpson problem for Coxeter connections}\label{meromorphicconnections}
\subsection{Moduli spaces of connections with toral singularities}
We now turn our attention to meromorphic connections $\gc$ on a rank $n$
trivializable vector bundle $V$ over the complex Riemann sphere
$\pp$. To discuss the Deligne--Simpson problem, we need to define what
it means for $\gc$ to be framable at a singularity with respect to a given toral formal type.
We will assume that the singular point is $0$.  The only modification
needed if the singularity is at an arbitrary point $a\in \pp$ is to replace the uniformizer
$z$ by $z-a$ if $a$ is finite and by $z^{-1}$ if $a=\infty$.

Fix a trivialization $\phi$ of $V$, and write $\gc=d+[\gc]_\phi$.  The
principal part of $[\gc]_\phi$ is an element of
$\fgl_n(\C[z^{-1}]) \tfrac{dz}{z}$, and so may be viewed as a
continuous functional on $\fgl_n(\pow)$ by \eqref{duals}.  Similarly,
the restriction of $[\gc]_\phi$ to $\fpara^b$ is uniquely determined
by the truncation of $[\gc]_\phi/\dzz$ to
$\bigoplus_{i=0}^\infty \fpara^b(-i)$.  Thus, if $\ftype=A\dzz$ is an
$\Stor{b}$-formal type, then $\ftype$ may naturally be viewed as an
element of $(\fpara^b)^\vee$.

\begin{definition}\label{D:formaltype} Let $\gc=d+[\gc]_\phi$ be a global connection on $\pp$ with a singular point at $0$,
and let $\ftype=A\dzz$ be an $\Stor{b}$-formal type of depth $r$. We say
that $\gc$ is \emph{framable at $0$ with respect to $\ftype$} if \begin{enumerate}\item  there
  exists $g\in \GL_n(\C)$ such that $[\gc]_{g\phi}=g\cdot [\gc]_\phi\in(\fpara^{b})^{-r}$ and
  $[\gc]-A\in(\fpara^{b})^{1-r}$, and
\item there exists an element $p\in (\Para^{b})^1$ such that $\Ad^*(p)([\gc]_{g\phi}\dzz)|_{\fpara^b}=\ftype$.
\end{enumerate}
\end{definition}

It is a consequence of Theorem~\ref{T:diag} that this definition is
equivalent to Definition~\ref{D:formaltypealt}.

Fix two disjoint subsets $\{a_1,\dots,a_m\}$ and
$\{b_1,\dots,b_\ell\}$ of $\pp$ with $m\ge 1$ and $\ell\ge 0$.  Let
$\bfA=(\ftype_1,\dots,\ftype_m)$ be a collection of toral formal types at the $a_i$'s,
and let $\bfO=(\orb_1,\dots,\orb_\ell)$ be a collection of adjoint
orbits at the $b_j$'s.  Assume that all of the orbits $\orb_j$ are \emph{nonresonant},
meaning that no two eigenvalues of an orbit differ by a nonzero
integer.   One can now consider the category $\sC(\bfA,\bfO)$ of meromorphic connections $\gc$ satisfying the following properties:
\begin{enumerate}
\item $\gc$ has irregular singularities at the $a_i$'s, regular
  singularities at the $b_j$'s, and no other singular points;
\item for each $i$, $\gc$ is framable at $a_i$ with respect to the formal type $\ftype_i$; and
\item  for each $j$, $\gc$ has residue at $b_j$ in $\orb_j$.
\end{enumerate}

 In \cite{BrSa1}, Bremer and Sage constructed the moduli
  space $\M(\bfA,\bfO)$ of this category as a Hamiltonian reduction of
  a product over the singular points of certain symplectic manifolds, each of
  which is endowed with a Hamiltonian action of $\GL_n(\C)$.  At a
  regular singular point with adjoint orbit $\orb$, the manifold is
  just $\orb$, viewed as the coadjoint orbit $\orb\dzz$.  To define the symplectic
  manifold $\M_\ftype$ associated to an $\Stor{b}$-toral formal type $\ftype$, we first
  remark that the parahoric subgroup $\Para^b$ is the pullback of a
  certain standard parabolic subgroup $Q^b\subset\GL_n(\C)$ under the
  map $\GL_n(\pow)\to\GL_n(\C)$ induced by $z\mapsto 0$.  For example, $\Para^1=\GL_n(\pow)$ is the pullback of $Q^1=\GL_n(\C)$, and $\Para^n=\Iwa$ is the pullback of $Q^n=\B$.  The
  ``extended orbit'' $\M_\ftype\subset
  (\Para^b\backslash\GL_n(\C))\times\fgl_n(\pow)^\vee$ is defined by
\begin{equation*}\M_{\ftype}=\{(Q^b g,\alpha) \mid
  (\Ad^*(g)\alpha)|_{\fpara^b}\in \Ad^*(\Para^b)(\ftype)\}.
\end{equation*}
The group $\GL_n(\C)$ acts on $\M_{\ftype}$ via
$h\cdot(Q^b g,\alpha)=(Q^b gh^{-1},\Ad^*(h)\alpha)$, with moment map
$(Q^b g, \alpha)\mapsto\alpha|_{\fgl_n(\C)}$.

\begin{thm}[{\cite[Theorem 5.26]{BrSa1}}]\label{genmodspace} The
  moduli space $\M(\bfA,\bfO)$ is given
  by 
\begin{equation*}  \M(\bfA,\bfO)\cong \left[ \left( \prod_i
      \M_{\ftype_i}\right) \times \left( \prod_j \orb_j \right)
  \right] \sslash_0 \GL_n(\C).
\end{equation*}
\end{thm}

Let $\M_{\mathrm{irr}}(\bfA,\bfO)$ be the subset of $\M(\bfA,\bfO)$
  consisting of irreducible connections.  One can now restate the
  toral Deligne--Simpson problem as
  \begin{quote}
\emph{Given the toral formal types $\bfA$ and the nonresonant adjoint
  orbits $\bfO$, determine whether $\M_{\mathrm{irr}}(\bfA,\bfO)$ is nonempty.}
\end{quote}

Note that a $\cC$-toral connection is irreducible, so if any
$\ftype_i$ is $\cC$-toral, then the Deligne--Simpson problem reduces
to the question of whether $\M(\bfA,\bfO)$ is nonempty.

\subsection{Coxeter connections}
We now specialize to an important special case: connections on $\Gm$
with a maximally ramified singular point at $0$ and (possibly) a
regular singularity at $\infty$.  Since maximally ramified formal
connections are Coxeter toral, we will follow \cite{KS2} and refer to
such connections as \emph{Coxeter connections}.

It is possible to give a simpler expression for moduli spaces of
Coxeter connections.

\begin{prop}
Let $\ftype$ be a $\cC$-formal type, and let $\orb$ be a nonresonant
adjoint orbit.  Then \begin{equation*} \M(\ftype,\orb) \cong
  \{(\alpha,Y) \mid \alpha\in\fgl_n(\C[z^{-1}])\tdzz, Y\in\orb,
  \alpha|_{\iwa}\in\Ad^*(\Iwa)(\ftype),\text{ and
  }\Res(\alpha)+Y=0\}/\B .
\end{equation*}
\end{prop}
\begin{proof} Applying Theorem~\ref{genmodspace}, we have
  \[\begin{array}{rll}
      \M(\ftype,\orb)
      &\cong(\M_{\ftype}\times\orb) \sslash_0 \GL_n(\C)\\
      &\cong\{(\B g,\alpha,Y) \mid (\B g,\alpha)\in\M_{\ftype},Y\in\orb,\text{ and }\Res(\alpha)+Y=0\}/\GL_n(\C)\\
      &\cong\{(\B,\alpha,Y) \mid \alpha\in\fgl_n(\C[z^{-1}])\dzz,
        Y\in\orb, \alpha|_{\iwa}\in\Ad^*(\Iwa)(\ftype),\text{ and
        }\Res(\alpha)+Y=0\}/\B.\end{array}\]
  \end{proof}

  We can now state the solution to the Deligne--Simpson problem for
  Coxeter connections.  For a given $\cC$-formal type and a monic
  polynomial $q$ of degree $n$, let \[\DS(\ftype,
    q)=\{\orb\in\orbp{q}\mid\M(\ftype, \orb) \neq \varnothing \}.\]

  \begin{thm}\label{thm:mainthm}
  Let $r$ and $n$ be positive integers with $\gcd(r,n)=1$, let $\ftype$ be a maximally ramified
  formal type of slope $r/n$, and let $q=\prod_{i=1}^s(x-a_i)^{m_i}\in\C[x]$ with $a_1,\ldots,a_s\in\C$ distinct modulo $\Z$.  Then
  \[
    \DS(\ftype, q)=
    \begin{cases}
      \prinfilt{\prp{r}{q}} & \text{if }\Res(\Tr(\ftype))=-\sum_{i=1}^sm_ia_i, \\
      \varnothing & \text{else}.
    \end{cases}
  \]

\end{thm}

We prove this theorem in the next subsection.

\begin{rmk}  If we write $\ftype=p(\om^{-1})\dzz$, then $\M(\ftype,\orb)$ is
nonempty if and only if $n p(0)=-\Tr(\orb)$ and if $\orb$ has at most
$r$ Jordan blocks for each eigenvalue.  This second condition is
always satisfied if $r>n$.  Note that the solution only depends on the slope and the residue of
the formal type.
\end{rmk}

This theorem immediately gives the corresponding result for
$\SL_n$-connections.  (In terms of a global trivialization, one may
view an $\SL_n$-connection on $\pp$ as an operator $d+X\dzz$ with
$X\in\fsl_n(\C[z,z^{-1}])$.)  In this case, maximally ramified formal
types are of the form $p(\om^{-1})\dzz$ with $p(0)=0$ and
$\Tr(\orb)=0$, so the trace condition becomes vacuous.  Accordingly,
we obtain the following solution to the Deligne--Simpson for Coxeter
$\SL_n$-connections.
\begin{corollary}  The moduli space of Coxeter $\SL_n$-connections with formal type of slope $r/n$ and adjoint orbit $\orb$ is
  nonempty if and only if $\orb\succeq
\prp{r}{\ch(\orb)}$.
\end{corollary}

The notion of a Coxeter $\G$-connection makes sense for any reductive
group with connected Dynkin diagram, and there is an analogue of the
Deligne--Simpson problem in this context.  We give a specific conjecture about
this problem (under the additional hypothesis of unipotent monodromy)
in the introduction.

\subsection{Proof of Theorem~\ref{thm:mainthm}}

We begin with some preliminaries on $\Iwa^1$-orbits in $\iwa^\vee$.
\begin{lem}\label{adomegamap}
For any $r,s\in\Z$ with $r$ relatively prime to $n$, the linear map
$\phi_{r,s}:\ft\to\fgl_n(\C)$ given by $\gamma\mapsto
[\gamma\om^{r-s},\om^{-r}]\om^s$ has image
$\ft\cap\fsl_n(\C)$.
\end{lem}
\begin{proof} Since
  $[\gamma\om^{r-s},\om^{-r}]=\gamma\om^{-s}-\om^{-r}\gamma\om^{r-s} =
  (\gamma-\om^{-r}\gamma\om^{r})\om^{-s}$, we have
  $\phi_{r,s}(\gamma)=\gamma-\om^{-r}\gamma\om^r.$ Note that
  $\om^{-r}\in N(\T)$ --- indeed, it is a monomial matrix which
  represents the $-r$th power of a Coxeter element in the Weyl group
  $\symmetric{n}$ --- so the entries in $-\om^{-r}\gamma\om^r\in\ft$
  are a reordering of the entries in $\gamma\in\ft$.  It follows that
  $\phi_{r,s}(\ft)\subset \ft\cap\fsl_n(\C)$.  To show equality, it
  suffices to check that $\dim \ker \phi_{r,s} =1$.  We have
  $\gamma \in \ker\phi_{r,s}$ if and only if $\gamma$ commutes with
  $\om^{-r}$.  Since $\gcd(r,n)=1$, the centralizers of $\om^{-r}$ and
  $\om$ coincide and equal $\C(\!(\om)\!)$.  Thus,
  $\ker(\phi_{r,s})=\C(\!(\om)\!)\cap \ft=\C \Id$.
\end{proof}

We can now give a convenient representative of certain coadjoint $\Iwa^1$-orbits in
$\iwa^\vee$.

\begin{prop}\label{functionalorbit}
  Let $r$ and $n$ be positive integers with $\gcd(r,n)=1$.  Suppose $\alpha\in\iwa^\vee$ is
  given by $\alpha=(a\om^{-r}+X)\dzz$ for some $a\in\C^*$ and
  $X\in\iwa^{-r+1}$.  Then there exists $g\in\Iwa^1$ such that
  \begin{equation}\label{normalform}\Ad^*(g)\alpha=(a
  \om^{-r}+\sum_{i=0}^{r-1}c_i e_{ii}\om^{-i})\tdzz.
\end{equation}
\end{prop}

\begin{proof} 
  All elements of $\iwa^{1}\dzz$ represent the zero functional on
  $\iwa$, so it suffices to show that if
  $\alpha\in (a \om^{-r}+\sum_{i=s+1}^{r-1}c_i
  e_{ii}\om^{-i}+\beta\om^{-s}+\iwa^{-s+1})\dzz$ for $1\le s\le r-1$,
  then there exists $\gamma\in\ft$ such that
  $\Ad^*(1+\gamma\om^{r-s})(\alpha)\in (a\om^{-r}
  +\sum_{i=s+1}^{r-1}c_i e_{ii}\om^{-i}+c_s e_{ss}\om^{-s}+\iwa^{-s+1})\dzz$.

Write $\beta=(b_1,\dots,b_n)$.  By Lemma~\ref{adomegamap}, there exists
$\gamma\in\ft$ such that \begin{equation*}\phi_{r,s}(a\gamma)=(-b_1, \ldots, -b_{s-1},\sum_{i\neq s}b_i,
  -b_{s+1},\ldots, -b_n).
\end{equation*}
Setting $c_s=\sum_{i}b_i$ gives
\begin{align*}\Ad^*(1+\gamma\om^{r-s})(\alpha)&\in (a\om^{-r}
                                                   +\sum_{i=s+1}^{r-1}c_i
                                                   e_{ii}\om^{-i}+\beta\om^{-s}+a[\gamma\om^{r-s},\om^{-r}]+\iwa^{-s+1})\tdzz\\
                                                 &= (a\om^{-r}
                                                   +\sum_{i=s+1}^{r-1}c_i
                                                   e_{ii}\om^{-i}+(\beta+\phi_{r,s}(a\gamma))\om^{-s}+\iwa^{-s+1})\tdzz\\
                                                 &= (a\om^{-r}
                                                   +\sum_{i=s+1}^{r-1}c_i
                                                   e_{ii}\om^{-i}+c_{s}e_{ss}\om^{-s}+\iwa^{-s+1})\tdzz.
\end{align*}
\end{proof}

  \begin{corollary}\label{residueresult}  Given $\alpha$ as in the
  proposition and $D\in\ft$ with $\Tr(D)=\Res(\Tr(\alpha))$,   there
  exists $h\in\Iwa^1$ such that
  $\Res(\Ad^*(h)\alpha)=\Res(a\om^{-r}\tdzz)+D$.
\end{corollary}

\begin{proof}

 Given $g$ as in \eqref{normalform}, we claim that
\begin{equation*}\Res(\Ad^*(g)\alpha)\in\Res(a\om^{-r}\tdzz)+\ft.
\end{equation*}

To see this, write $s=kn+u$ with $0< u\le n$.  Recall that $N_u$ is
the matrix with $1$'s on the $u$th subdiagonal and $0$'s elsewhere.
Similarly, let $E_u$ be the matrix whose only nonzero entries are
$1$'s on the $(n-u)$th superdiagonal.  (We make the convention that
$N_n=0$ and $E_n=\Id$.)  It is easy to verify that
$\om^{-s}=z^{-k}(N_{u}+z^{-1}E_{u})$.  Since $e_{uu}E_u=e_{un}$, and
$e_{uu}N_u=0$, we have $e_{ss}\om^{-s}=z^{-(k+1)}e_{sn}$.  In
particular, $\Res(e_{ss}\om^{-s}\dzz)=0$ if $s>0$ and equals $e_{nn}$
if $s=0$.  Applying this to \eqref{normalform} gives
$\Res(\Ad^*(g)\alpha)=\Res(a\om^{-r}\dzz)+c_0e_{nn}$ as desired.

To complete the proof, it suffices to show that if
$\Res(\alpha)=\Res(a\om^{-r}\dzz)+D^\prime$ for some
$D^\prime\in\ft$, and if $D\in\ft$ satisfies $\Tr(D)=\Tr(D^\prime)$,
then there exists $\gamma\in\ft$ such that
$\Res(\Ad^*(1+\gamma\om^{r})\alpha)=\Res(a\om^{-r}\dzz)+D$.
Since
$\Res(\Ad^*(1+\gamma\om^{r})\alpha)=\Res(a\om^{-r}\dzz)+D^\prime+a[\gamma\om^{r},\om^{-r}]$,
it suffices to find $\gamma\in\ft$ such that
$D^\prime+a[\gamma\om^{r},\om^{-r}]=D$.  This follows from
Lemma~\ref{adomegamap}, since $a^{-1}(D-D^\prime)\in \ft\cap\fsl_n(\C) =\im(\phi_{r,0})$.
\end{proof}

\begin{lem}\label{lemmaatmostr}
Let $r$ and $n$ be positive integers with $\gcd(r,n)=1$. Suppose that $\ftype\in\ftypes(\cC,r)$
and $\orb$ is a nonresonant adjoint orbit in $\fgl_n(\C)$.  If $\M( \ftype, \orb) \neq \varnothing$, then the Jordan form of $\orb$ has at most $r$ blocks for each eigenvalue.
\end{lem}
\begin{proof} This is trivial for $r>n$, so assume that $r<n$.  Choose
  $\alpha\in \fgl_n(\C[z^{-1}])\dzz$ and $Y\in\orb$ such that
  $\alpha|_{\iwa}\in\Ad^*(\Iwa)(\ftype)$ and
  $Y=-\Res(\Ad^*(b)\alpha)$.  Write $\ftype=(a\om^{-r}+X)\dzz$ for
  some $a\in\C^*$ and for some $X\in\fcox^{-r+1}$.  Since
  $\Iwa=\T\Iwa^1$, we may assume without loss of generality that
  $\alpha|_{\iwa}\in\Ad^*(\Iwa^1)\ftype$.  This implies that
  $\alpha=(a\om^{-r}+X')\dzz$ for some $X'\in\iwa^{-r+1}$.  It is
  easy to see that $\Res(X'\dzz)$ has $0$'s in the $r$th
  subdiagonal and below.  In the notation of \S\ref{prelim}, this means that $Y\coloneqq -\Res(\alpha)\in V^r$.  By Theorem~\ref{thm:r-regular}, $Y$ has at most $r$ blocks for each eigenvalue.
\end{proof}

\begin{lem}\label{lemmacontains}
  Let $r$ and $n$ be positive integers with $\gcd(r,n)=1$, let
  $\ftype$ be a $\cC$-formal type of depth $r$, and let
  $q=\prod_{i=1}^s(x-a_i)^{m_i}\in\C[x]$ be a degree $n$ polynomial with $a_1,\ldots,a_s\in\C$ distinct modulo $\Z$.  If $\Res(\Tr(\ftype))=-\sum_{i=1}^sm_ia_i$, then $\prp{r}{q} \in \DS(\ftype, q)$.
\end{lem}

\begin{proof}

  Write $\ftype=(a\om^{-r}+X)\dzz$ for
  some $a\in\C^*$ and for some $X\in\fcox^{-r+1}$.  By
  Proposition~\ref{Jordanform}, there exists $D$ with trace
  $\sum_{i=1}^sm_ia_i$ such that $-aN_r+D\in\prp{r}{q}$.  We now apply
  Corollary~\ref{residueresult} to obtain $g\in\Iwa^1$ such that
  $\Res(\Ad^*(g)\ftype)=aN_r-D$.  Let $\alpha\in\fgl_n(\pow)^\vee$ be
  any functional
  extending $\Ad^*(g)\ftype\in\iwa^\vee$.  Then the $\B$-orbit of
  $(\alpha,-aN_r+D)$ gives a point in $\M(\ftype,\prp{r}{q})$,
  so $\prp{r}{q}\in\DS(\ftype,q)$.
\end{proof}

\begin{proof}[Proof of Theorem~\ref{thm:mainthm}] If $\M( \ftype,
  \orb) \neq \varnothing$, then there exists
  $\alpha\in\fgl_n(\pow)^\vee$ and $Y\in\orb$ with
  $\Res(\alpha)+Y=0$.  Since $\Res(\Tr(\alpha))=\Res(\Tr(\ftype))$, we
  see that $\Res(\Tr(\ftype))=-\Tr(Y)=-\sum_{i=1}^sm_ia_i$.

  By Lemma~\ref{lemmacontains},
  $\prp{r}{q} \in \DS(\ftype, q)$, and by Lemma~\ref{lemmaatmostr}, $\DS(\ftype, q)
  \subset\prinfilt{\prp{r}{q}}$. To show equality, take
  $\orb\in\prinfilt{\prp{r}{p}}$.  Since $\prp{r}{q} \in \DS(\ftype,
  q)$, there exists some $X\in\fgl_n(\C[z^{-1}])$ and $Y\in\prp{r}{q}$
  such that $(X\dzz)|_{\iwa}\in\Ad^*(\Iwa)(\ftype)$ and
  $\Res(X\dzz)+Y=0$.  By a theorem of Krupnik~\cite[Theorem 1]{krupnik97}, there exists a
  strictly upper triangular matrix $N\in\fgl_n(\C)$ such that
  $Y+N\in\orb$.  Note that $N\subset\iwa^1$, so $(N\dzz)|_{\iwa}=0$.
  Thus, $(X-N)\in\fgl_n(\C[z^{-1}])$,
  $((X-N)\dzz)|_{\iwa}\in\Ad^*(\Iwa)(\ftype)$, and
  $\Res((X-N)\dzz)+(Y+N)=0$.  Hence $\M(\ftype, \orb) \neq
  \varnothing$, and the proof is finished.
\end{proof}

\begin{rmk}  At least in the case of unipotent monodromy, it is
  possible to avoid using Krupnik's Theorem by giving an explicit
  construction of an element of the moduli space.  We discuss this in
  ~\cite{KLMNS2}.
\end{rmk}

\section{Rigidity for Coxeter connections}\label{s:rigidity}
Let $\gc$ be a meromorphic $\G$-connection on $\pp$ which is regular on the
Zariski-open set $U=\pp\setminus\{x_1,\dots,x_k\}$.  Let $\fc_{x_i}$
denote the induced formal $\G$-connection at $x_i$.  The connection is called
\emph{physically rigid} if, for any meromorphic $\G$-connection $\gc'$ which is regular
on $U$ and satisfies $\fc'_{x_i}\cong\fc_{x_i}$ for all $i$, we have
$\gc'\cong\gc$.

In general, it is very difficult to determine whether a connection is
physically rigid.  A more accessible notion is given by
\emph{cohomological rigidity}, which means that
$H^1(\pp,j_{!*}\ad_{\gc})=0$, where $j:U\hookrightarrow\pp$ is the
inclusion.  If $\gc$ is irreducible, then $\gc$ being cohomologically rigid
implies that $\gc$ admits no infinitesimal
deformations~\cite{Yun14}.  For $\G=\GL_n(\C)$, Bloch and Esnault have shown
that cohomological rigidity and physical rigidity are equivalent~\cite{BE04}.

We call a $\cC$-formal type \emph{homogeneous} when it is
of the form $a\om^{-r}\dzz$ for $a\in\C^*$; it gives rise to a
``homogeneous Coxeter connection'' $d+a\om^{-r}\dzz$ on $\Gm$.
This connection has a toral singularity at $0$ and (possibly) a
regular singularity at $\infty$ with unipotent monodromy.  This notion also makes sense for any
complex simple group $\G$~\cite{KS2}.  Again, one can define formal types with
respect to a certain maximal
torus $\cC_{\G}\subset \G(F)$ called the Coxeter torus.  Moreover, if $r$ is any
positive integer relatively prime to the Coxeter number $h$, there
exists  an element $\om_{-r}\in\Lie(\cC_\G)$ such that
$a\om_{-r}\dzz$ may be viewed as a homogeneous formal type.  One
can again consider the corresponding Coxeter $\G$-connection on $\Gm$
with a homogeneous $\cC_\G$-toral irregular singularity of
slope $r/h$ at $0$ and (possibly) a regular singular point with
unipotent monodromy at $\infty$.  The case $r=1$ is the remarkable rigid connection
constructed by Frenkel and Gross~\cite{FrGr}.

In \cite{KS2}, Kamgarpour and Sage determined when these homogeneous
Coxeter $\G$-connections are (cohomologically) rigid for any simple $\G$.  It turns out
$r=1$ and $r=h+1$ always give rigid connections: the Frenkel--Gross and
``Airy $\G$-connection'' respectively.  For the exceptional groups,
there are no other such rigid connections except for $r=7$ in $E_7$.
However, for the classical groups, one
also has rigidity for $1<r<h$ with $r$ satisfying certain divisibility
conditions.  For example, in type $A$, these connections are rigid if
and only if $r|(n\pm 1)$.

In this paper, we generalize this result in type $A$ to give a
classification of rigid framable Coxeter connections with unipotent
monodromy at $\infty$.  (Framable means that we only consider Coxeter
connections contained in the relevant framable moduli space $\M(\ftype,\orb)$.)

 \begin{thm}\label{thm:rigidthm}
Let $\ftype$ be a rank $n$ maximally ramified formal
  type of slope $r/n$, and let $\orb$ be any nilpotent orbit with
  $\orb\succeq\orb^r_{x^n}$.  Then there exists a rigid connection
  with the given formal type and unipotent monodromy determined by $\orb$ if and only if
  $\orb=\orb^r_{x^n}$ and $r|(n\pm 1)$.
\end{thm}

If $\gc$ is a Coxeter connection, let $\Galgp$ denote the global
differential Galois group, and let $\Galgp_0$ and $\Galgp_\infty$ denote the
local differential Galois groups at $0$ and $\infty$.   These
differential Galois groups are all algebraic subgroups of
$\GL_n(\C)$.  Also, let $\Irr(\ad_{\fc_0})$ denote the irregularity of
the formal connection $\ad_{\fc_0}$.  This is the sum of all the
slopes appearing in the slope decomposition of $\ad_{\fc_0}$; it is a
nonnegative integer.

Let $n(\gc)=\Irr(\ad_{\fc_0})-\dim(\fgl_n(\C)^{\Galgp_0})
-\dim(\fgl_n(\C)^{\Galgp_\infty})+2\dim(\fgl_n(\C)^{\Galgp})$, and let
$j:\Gm\hookrightarrow\pp$ be the inclusion.  It is shown
in~\cite[Proposition 11]{FrGr} that $\dim(H^1(\pp,j_{!*}\ad_{\gc}))=n(\nabla)$.  Thus,
we get the \emph{numerical criterion for rigidity} that $\gc$ is rigid
if and only if $n(\gc)=0$.

We now calculate the numerical criterion as in Section 4 of \cite{KS2}.  The local differential Galois group $\Galgp_0$ is given by $\Galgp_0\cong H\ltimes \langle\theta\rangle$, where $H$ is a
certain torus containing a regular semisimple element and $\theta$ is
an order $n$ element of $N(H)$~\cite{KS3}.  The centralizer of $H$ is thus a
maximal torus $\T'$, and $\theta\in N(\T')$ represents a Coxeter element
in the Weyl group.  We conclude (as in \cite{KS2}) that
$\fgl_n(\C)^{\Galgp_0}=(\fgl_n(\C)^H)^\theta=\Lie(T')^\theta=\C \Id$.  Since
$\Galgp_0\subset \Galgp$, we also have $\fgl_n(\C)^{\Galgp}=\C \Id$.

In general, if $\fc$ is a toral $\G$-connection with slope $s$, then by Lemma 19 of~\cite{KS1},
$\Irr(\ad_{\fc})=s|\Phi|$, where $\Phi$ is the set of roots with respect to
the maximal torus $T$.  In our particular case, we obtain
$\Irr(\ad_{\fc_0})=\frac{r}{n}n(n-1)=r(n-1)$.

Finally, if we fix some element $N_\orb\in\orb$, then $\fc_\infty$ is
regular singular with unipotent monodromy $\exp(2\pi i N_\orb)$.  This
means that $\Galgp_\infty\cong \langle\exp(2\pi i N_\orb)\rangle$, so
$\fgl_n(\C)^{\Galgp_\infty}=C(N_\orb)$, the centralizer of $N_\orb$.
Using the fact
that $\dim(\orb)=n^2-\dim(C(N_\orb) )$, we obtain
\begin{equation*} n(\gc)=(r-n-1)(n-1)+\dim(\orb).
\end{equation*}

Since $\M(\ftype, \prp{r}{x^n})\neq\varnothing$, we can take $\gc'$
with this local behavior.  Suppose that $\orb\succneqq \prp{r}{x^n}$.
We then have
$n(\gc)=(r-n-1)(n-1)+\dim(\orb)>
(r-n-1)(n-1)+\dim(\prp{r}{x^n})=n(\gc')=0$.  It follows that in this
case, $\gc$ is never rigid.

Finally, take $\orb=\prp{r}{x^n}$.  It was shown in \cite{KS2} that
$\dim(\prp{r}{x^n})=(n+1-r)(n-1)$ if and only if $r|(n\pm 1)$.  This
finishes the proof of the theorem.


\bibliography{references}
\bibliographystyle{amsalpha}

\end{document}